\documentclass{article}
\usepackage{graphicx} 
\usepackage{xcolor}
\usepackage{palatino}
\usepackage{amsthm}
\usepackage{mathrsfs}
\usepackage{amssymb}
\usepackage[
backend=biber,
    sorting=nyt,       
    style=alphabetic,
    maxcitenames=99,   
    maxalphanames=99,
    mincitenames=99,
    maxbibnames=99,    
    minbibnames=99
]{biblatex}
\usepackage[hidelinks]{hyperref}
\usepackage{amsmath}
\usepackage{mathtools}
\usepackage{amsfonts}
\usepackage{tikz-cd}  
\usepackage{textcomp}
\usepackage{enumerate}
\newtheorem{theorem}{Theorem}
\newtheorem{proposition}{Proposition}
\newtheorem{lemma}{Lemma}
\newtheorem{remark}{Remark}
\newtheorem{definition}{Definition}

\newtheorem*{theoremA}{Theorem A}
\newtheorem*{theoremB}{Theorem B}
\newcommand{\hilb}{\mathcal{H}}

\newcommand{\B}{\mathcal{B}}
\newcommand{\C}{\mathbb{C}}
\newcommand{\R}{\mathbb{R}}

\newcommand{\A}{\mathcal{A}}
\newcommand{\Z}{\mathbb{Z}}

\title{\huge{Discrete $C^*$-Frobenius Algebras and Infinite-Index Extensions in Algebraic Quantum Field Theories}}
\author{{\sc ZIYUN XU}\\{\small Graduate School of Mathematical Sciences}\\{\small The University of Tokyo, Komaba, Tokyo, 153-8914, Japan}\\{\small e-mail: {\tt zyxu0805@g.ecc.u-tokyo.ac.jp}}}

\begin{document}
\maketitle
\begin{abstract}

We introduce discrete $C^*$-Frobenius algebras as categorical data for constructing infinite-index extensions of M\"{o}bius covariant nets. The underlying representation is a countable direct sum of simple dualized modules, while the direct sum itself is not necessarily dualizable. Accordingly, rather than requiring a single bounded multiplication morphism, we work with compatible families of bounded left and right multiplication morphisms.
Using Bin Gui's theory of categorical extensions, we associate to every discrete
$C^*$-Frobenius algebra two (generally non-local) M\"{o}bius covariant extensions generated by left and right charged fields, respectively. If the discrete $C^*$-Frobenius algebra is commutative, the two extensions coincide and yield a local M\"{o}bius covariant extension. 
As applications,  for every $C^*$-tensor category generated by an invertible object with trivial self-braiding such that its fusion product rule is given by  $\Z$, 
we construct two classes of examples: the simple current extension by $\Z$ and the discrete Longo-Rehren construction.

\end{abstract}

\section{Introduction}

The theory of operator algebras plays an important role in the mathematical study of two-dimensional rational
conformal field theories. For a local M\"{o}bius covariant net $\A$, extensions with finite index are described by 
$Q$-systems \cite{longo1994duality} or special $C^*$-Frobenius algebras \cite{fuchs2003category} in $\operatorname{Rep}(\A)$. In the Connes fusion approach, 
this relation has been developed in Gui's work \cite{gui2021bisognano,gui2025connes}. It gives a precise procedure to construct the extension 
from the categorical data describing its charged sectors. 

The finite index assumption is restrictive when one moves to non-rational conformal field theories, for example, the non-compact free boson theory. Infinite
index extensions arise naturally and their operator algebraic study requires structures beyond 
ordinary $C^*$-Frobenius algebras.
Let $\hilb \cong \bigoplus_{i\in\Z}\hilb_i$ 
be a countably infinite direct sum of dualized modules of $\A$. $\hilb$ is generally non-dualizable. Therefore,
the ordinary $C^*$-Frobenius algebra formalism does not apply directly. Moreover, 
the formal multiplication map obtained by summing over infinitely many modules is not necessarily bounded.

Infinite-index extensions have previously been studied through generalized $Q$-systems of intertwiners by Del Vecchio and Giorgetti \cite{del2018infinite}, using the Doplicher-Haag-Roberts (DHR) theory \cite{fredenhagen1989superselection,fredenhagen1992superselection}. The infinite-depth version of the Longo-Rehren construction \cite{longo1995nets} at 
the level of subfactors has been studied by Masuda in \cite{masuda2000generalization}. 
The perspective of the 
present paper is complementary: we formulate infinite-index extensions in the framework of Connes fusion,
although the underlying strategy closely parallels that of Del Vecchio and Giorgetti \cite{del2018infinite}.

Our main idea is to replace the global multiplication morphism by sectorwise bounded left and right multiplication operators. More precisely, we introduce a discrete $C^*$-Frobenius algebra whose 
underlying module is $\hilb \cong \bigoplus_{i\in\Lambda}\hilb_i$,  a countably infinite direct sum of dualized modules of $\A$ and 
whose multiplication is encoded by bounded maps $\left\{ \mu_i^{\mathfrak{L}}: \hilb_i\boxtimes\hilb \to \hilb \right\}_{i\in\Lambda}$ and $\left\{ \mu_i^{\mathfrak{R}} : \hilb\boxtimes \hilb_i \to \hilb \right\}_{i\in\Lambda}$, where $\Lambda$ is a set with countably many elements.
The left and right families satisfy the unit condition, associativity and Frobenius-type compatibility conditions. 
This formulation avoids requiring a single bounded multiplication 
$\mu: \hilb\boxtimes\hilb \to \hilb$, while retaining the relations needed to construct local algebras of the extensions.

By using Gui's categorical extensions \cite{gui2021categorical}, every discrete \(C^*\)-Frobenius algebra gives rise to two, in general distinct, M\"{o}bius covariant extensions. When the discrete 
$C^*$-Frobenius algebra is commutative, which is defined in terms of braidings in $\operatorname{Rep}(\A)$, we show the two extensions actually coincide and the common family defines a 
local M\"{o}bius covariant extension. Our first main result can be summarized as follows.

\begin{theoremA}[Theorem \ref{relative local extensions} and Theorem \ref{localextension}]
Let
\[
    \left(
    \mathcal H \cong \bigoplus_{i\in\Lambda}\mathcal H_i,
    \{\mu_i^\mathfrak{L}\}_{i\in\Lambda},
    \{\mu_i^\mathfrak{R}\}_{i\in\Lambda},
    \{T_i\}_{i\in\Lambda}
    \right)
\]
be a discrete \(C^*\)-Frobenius algebra in
\(\operatorname{Rep}(\mathcal A)\).
Then the left and right charged fields define two M\"{o}bius covariant
extensions \(\mathcal B^\mathfrak{L}\) and \(\mathcal B^\mathfrak{R}\) of \(\mathcal A\).
If the discrete \(C^*\)-Frobenius algebra is commutative, then
the two extensions coincide, and the common family of von Neumann algebras defines a local M\"{o}bius covariant extension of \(\mathcal A\).
\end{theoremA}

We give two families of examples. First, 
assume $\left\{ \hilb_i\right\}_{i\in\Z}$ is a family of mutually inequivalent $\A$-modules with quantum dimension one and that the self-braiding 
of the generating module $\hilb_1$ is trivial. We construct compatible families of unitary intertwiners and bounded left and right multiplication maps. These yield a commutative discrete $C^*$-Frobenius algebra and hence an infinite index local M\"{o}bius covariant extension, the simple current extension by $\Z$.
The simple current extensions by $\Z$ were studied in 
\cite[Section 3]{staszkiewicz1995lokale} using the DHR theory. 
Our second application is a discrete Longo-Rehren construction.
The two applications can be summarized as follows:

\begin{theoremB}[Theorem \ref{simplecurrentextension} and Theorem \ref{discreteLRconstruction}]
Let \(\{\mathcal H_n\}_{n\in\mathbb Z}\) be a family of mutually
inequivalent invertible simple objects of \(\operatorname{Rep}(\mathcal A)\)
with fusion rules
\[
    \mathcal H_n \boxtimes \mathcal H_m
    \cong
    \mathcal H_{n+m},
    \qquad n,m\in\mathbb Z,
\]
and assume that the self-braiding of \(\mathcal H_1\) is trivial.
Then the following hold:

\begin{enumerate}
    \item
    The $\A$-module
    \[
        \bigoplus_{n\in\mathbb Z}\mathcal H_n
    \]
    serves as the underlying Hilbert space of  a commutative discrete \(C^*\)-Frobenius algebra in $\operatorname{Rep}(\A)$, whose
    associated net is the local simple-current extension by
    \(\mathbb Z\).

    \item
    On the $\A\otimes \A^{\mathfrak{r}}$-module
    \[
        \mathcal H_\bullet
        =
        \bigoplus_{n\in\mathbb Z}
        \mathcal H_n \otimes \Theta_n \mathcal H_n ,
    \]
    we can construct a two-dimensional local M\"{o}bius covariant net $\B$ extending $\A\otimes \A^{\mathfrak{r}}$, which can be regarded as a
    discrete Longo-Rehren construction.
\end{enumerate}
\end{theoremB}

One motivation for the latter construction comes from the boundary-bulk (or open-closed) conformal field theories \cite{cardy1986effect}. In the rational setting, the Longo-Rehren and full-center constructions (or bulk constructions) \cite{Fje+08,kong2008morita} relate the chiral or boundary theories to the bulk theories. An operator-algebraic formulation of this picture was developed in \cite{bischoff2015characterization,bischoff2015tensor}. In particular, the relation between discrete Longo-Rehren constructions and categorical full centers is well understood under finiteness assumptions.
For non-rational theories one expects infinite families of sectors, so the finiteness and dualizability assumptions underlying the standard constructions are no longer automatic. The present paper does not attempt to give general full center constructions for non-rational conformal field theories. Rather, it provides a concrete operator-algebraic framework in which an infinite-index analogue of the Longo-Rehren construction can be carried out.  We expect this framework to be useful in the future study of extensions and bulk constructions for many non-rational theories.

This paper is organized as follows.
In Section 2, we recall local M\"{o}bius covariant nets and categorical extensions. In Section 3, we introduce
discrete $C^*$-Frobenius algebras and construct M\"{o}bius covariant extensions. Section 4 treats the commutative case and proves locality.
Section 5 constructs the simple current extensions by $\Z$. 
Section 6 gives the discrete Longo--Rehren construction and the resulting two-dimensional local Möbius covariant net.

\tableofcontents

\section{Preliminaries}
\subsection{M\"{o}bius covariant nets}
Let $\mathcal{I}$ be the set of all non-empty non-dense open intervals in the unit circle $S^1$. If $I\in\mathcal{I}$, 
$I^c$ denotes the interior of the complement of $I$. 

A \textbf{(local) M\"{o}bius covariant net} $\A$  is a family of von Neumann algebras $\left\{ \A(I)\right\}$ indexed by $I \in \mathcal{I}$ on a fixed separable Hilbert space $\hilb_0$, satisfying the following axioms:
\begin{enumerate}[(a)]
    \item (Isotony) If $I \subset J$, then $\A(I) \subset \A(J). $
    \item (Locality) If $I \cap J = \emptyset,$ then $[\A(I), \A(J)] = \{0\}$, where the bracket denotes the commutator. 
    \item (M\"{o}bius covariance) There is a strongly continuous  unitary representation $U$ of $PSU(1,1)$ on $\hilb_0$, such that 
    \[
    U(g) \A(I) U(g)^* = \A(gI),
    \]
    for any $g\in PSU(1,1), I \in \mathcal{I}$. 
   \item (Positive energy) The generator of the rotation subgroup is positive.
   \item (Vacuum) There exists a $PSU(1,1)$-invariant unit vector $\Omega$, called the \textbf{vacuum vector}, such that 
   $\Omega$ is cyclic for $\bigvee_{I\in\mathcal{I}}\A(I)$ (the von Neumann algebra generated by all $\A(I)$).
\end{enumerate}

If $\A$ is a M\"{o}bius covariant net, it satisfies the following properties (see \cite{guido1996conformal} and the references therein):
\begin{enumerate}
    \item (Additivity) $\A(I) = \bigvee_i \A(I_i)$, where $\left\{ I_i\right\}$ is a set of open intervals such that 
    $\bigcup_i I_i = I$. 
    \item (Haag Duality) For any $I\in\mathcal{I}$, $A(I)' = \A(I^c)$.
    \item (Reeh-Schlieder theorem) For any $I\in\mathcal{I}$, $\A(I)\Omega$ is dense in $\hilb_0$. 
    \item For any $I \in \mathcal{I}$, $\A(I)$ is a type III factor. 
\end{enumerate}

An \textbf{$\A$-module} (or a representation of $\A$) consists of a separable Hilbert space 
$\hilb_i$ and a family of normal unital $*$-representations $\pi_{i,I}: \A(I) \to B(\hilb_i)$, for $I\in\mathcal{I}$, 
such that for any $I, J \in \mathcal{I}$ and $I\subset J$, 
$\pi_{i,J} \restriction_{\A(I)} = \pi_{i,I}$. We will write $\pi_{i,I}$ simply as $\pi_i$ when there is no ambiguity.
$\hilb_0$ is automatically an $\A$-module, called the \textbf{vacuum module}. An $\A$-module $\hilb_i$
is \textbf{ M\"{o}bius covariant} if  there is a strongly continuous unitary representation 
$U_i$ of $\widetilde{PSU}(1,1)$ on $\hilb_i$ such
that 
\[U_i(g)\pi_{i,I}(x)U_i(g)^* = \pi_{i,gI}(U(g)xU(g)^*), \] for any $g\in\widetilde{PSU}(1,1)$, $I\in\mathcal{I}$ and $x\in\A(I)$.

A local M\"{o}bius covariant net $\A$ is said to be \textbf{conformal covariant} if 
$U$ extends to a strongly continuous projective unitary representation of the orientation-preserving diffeomorphisms of  $S^1$, 
$\operatorname{Diff}^+(S^1)$ on $\hilb_0$, such that for any $g\in \operatorname{Diff}^+(S^1)$, $I\in \mathcal{I}$ and any
representing element $U_g \in U(\hilb_0)$,
\[
U_g \A(I) U_g^* = \A(gI). 
\]
Let $\mathcal{G} = \widetilde{\operatorname{Diff}^+}(S^1)$ be  the simply connected covering group of ${\operatorname{Diff}^+}(S^1)$. Define a topological group:
\[
\mathcal{G}_\A =  \left\{(g,V)\in \mathcal{G} \times U(\hilb_0): V \ \text{is a representing element of} \ U(g)\right\},
\]
called the central extension of $\mathcal{G}$ associated with $\A$. 
Then we have a representation $U: \mathcal{G}_\A \to B(\hilb_0)$ defined by $U(g,V) = V$.  Any $\A$-module $\hilb_i$ is \textbf{ conformal covariant } in the sense that there is a unique unitary representation 
$U_i$ of $\mathcal{G}_\A$ on $\hilb_i$ such that 
\[U_i(g)\pi_{i,I}(x)U_i(g)^* = \pi_{i,gI}(U(g)xU(g)^*), \] for any $g\in\mathcal{G}_\A$, $I\in\mathcal{I}$ and $x\in\A(I)$. See \cite{henriques2019loop,gui2021categorical} and references therein for more details.

\subsection{Categorical Extensions}
We recall the categorical extension introduced by \cite{gui2021categorical}. We refer the reader to \cite{gui2021categorical,gui2021bisognano} and references therein for more details.

Let $\operatorname{Rep}(\A)$ be the $C^*$-tensor category of $\A$-modules whose monoidal structure is given by Connes fusion \cite{bartels2015conformal,bartels2017conformal,gui2021categorical}, which is equivalent to using Doplicher-Haag-Roberts (DHR) superselection theory \cite{fredenhagen1989superselection,fredenhagen1992superselection} by \cite[Chapter 6]{gui2021categorical}. The unit object 
is $\hilb_0$. The tensor (fusion) product of $\hilb_i$ and $\hilb_j$ is written as $\hilb_i \boxtimes \hilb_j$.  
We assume without loss of generality that
$\operatorname{Rep}(\A)$ is strict, which means that we will not distinguish between
$\hilb_i, \hilb_0 \boxtimes \hilb_i, \hilb_i \boxtimes \hilb_0$ or $(\hilb_i\boxtimes\hilb_j)\boxtimes\hilb_k$  and $\hilb_i\boxtimes(\hilb_j\boxtimes\hilb_k)$ denoted as $\hilb_i\boxtimes\hilb_j\boxtimes\hilb_k$. 

If $I\in\mathcal{I}$, the arg-function of $I$ is a continuous function $\arg_I: I \to \R$ such that 
for any $e^{i\theta} \in S^1$, $\arg_I(e^{i\theta}) - \theta \in 2\pi\Z$. $\widetilde{I} = (I, \arg_I)$ is called an 
\textbf{arg-valued} interval. In other words, $\widetilde{I}$ is a branch of $I$ in the universal cover of $S^1$. 
Denote by $\widetilde{\mathcal{I}}$ the set of arg-valued intervals. 
$\widetilde{I}=(I, \arg_I)$ and $\widetilde{J}=(J,\arg_J)$ are said to be \textbf{disjoint} if $I $ and $J$ are disjoint. $\widetilde{I}$ is said to be \textbf{anticlockwise} to $\widetilde{J}$(or equivalently, $\widetilde{J}$ is \textbf{clockwise} to $\widetilde{I}$), if $I$ and $J$ are disjoint and $\arg_J(w)< \arg_I(z) <\arg_J(w)+2\pi$ for any 
$z\in I, w\in J$. For $\widetilde{I}\in\widetilde{\mathcal{I}}$, define $\widetilde{I}' = (I', \arg_{I'}) \in \widetilde{\mathcal{I}}$ such that $\widetilde{I}$ is anticlockwise to $\widetilde{I}'$. $\widetilde{I}'$ is called the \textbf{clockwise complement} of $\widetilde{I}$.

Let $\hilb_i, \hilb_j$ be $\A$-modules and $I\in\mathcal{I}$. Denote 
by $\operatorname{Hom}_{\A(I')}(\hilb_i,\hilb_j)$ the vector space of bounded linear operators
$X: \hilb_i \to \hilb_j$ such that 
$X\pi_{i,I'}(a) =  \pi_{j,I'}(a)X$ for any $a\in\A(I')$. 
We say a vector $\xi\in\hilb_i$ is \textbf{$I$-bounded} if there exists
$Z(\xi, I)\in\operatorname{Hom}_{\A(I^c)}(\hilb_0, \hilb_i)$ such that 
$Z(\xi, I)\Omega = \xi$. By the Reeh-Schlieder theorem, the operator is unique whenever it exists. Denote by $\hilb_i(I)$ the space of all $I$-bounded vectors in $\hilb_i.$

Let $\hilb_1, \hilb_2, \hilb_3, \hilb_4$ be Hilbert spaces and let
$A: \hilb_1 \to \hilb_2, \  T:\hilb_2 \to \hilb_4, \ S: \hilb_1 \to \hilb_3, \ B:\hilb_3 \to \hilb_4$ be bounded linear 
operators. 
We say the diagram
\[
\begin{tikzcd}
    \hilb_1 \arrow{r}{A} \arrow[swap]{d}{S}  &\hilb_2 \arrow{d}{T} \\
    \hilb_3  \arrow[swap]{r}{B} &\hilb_4
\end{tikzcd}
\]
\textbf{commutes adjointly} if both this diagram and the following diagram 
\[
\begin{tikzcd}
    \hilb_1 \arrow{r}{A}   &\hilb_2  \\
    \hilb_3  \arrow[swap]{r}{B}\arrow{u}{S^*} &\hilb_4 \arrow[swap]{u}{T^*}
\end{tikzcd}
\] commute. Note that the above diagram commutes if and only if the following diagram 
\[
\begin{tikzcd}
    \hilb_1  \arrow[swap]{d}{S}  &\hilb_2 \arrow{d}{T}\arrow[swap]{l}{A^*} \\
    \hilb_3   &\hilb_4 \arrow{l}{B^*}
\end{tikzcd}
\] commutes. Note also that if both $A$ and $B$ are unitary or both  $S$ and $T$ are unitary, commutativity implies adjoint commutativity. 

\begin{definition}[\cite{gui2021categorical}]
    The \textbf{categorical extension} $\mathfrak{E}=(\A, \operatorname{Rep}(\A), \boxtimes , \hilb)$ of $\A$ assigns bounded linear operators 
    \begin{align*}
        L(\xi, \widetilde{I}) &\in \operatorname{Hom}(\hilb_j, \hilb_i \boxtimes \hilb_j), \\
        R(\xi, \widetilde{I}) &\in \operatorname{Hom}(\hilb_j, \hilb_j\boxtimes  \hilb_i),
    \end{align*}
    to any $\hilb_i, \hilb_j \in \operatorname{Obj}(\operatorname{Rep}(\A))$, $\widetilde{I}\in \widetilde{\mathcal{I}}$ and $\xi \in \hilb_i(I)$, such that the following conditions are satisfied:
    \begin{enumerate}[A]
        \item (Isotony) If $\widetilde{I_1} \subset \widetilde{I_2}$ in $\widetilde{\mathcal{I}}$, then $L(\xi,\widetilde{I_1}) = L(\xi, \widetilde{I_2})$ and $R(\xi,\widetilde{I_1}) = R(\xi, \widetilde{I_2})$ acting on any object of $\operatorname{Rep}(\A)$. 
        \item (Functoriality) Let $\hilb_i, \hilb_j, \hilb_k \in \operatorname{Obj}(\operatorname{Rep}(\A)), \widetilde{I}\in \widetilde{\mathcal{I}}, F \in \operatorname{Hom}_{\A}(\hilb_j,\hilb_k), \xi\in\hilb_i(I), \eta\in\hilb_j$. The following relations hold:
        \[
        (\operatorname{id}_i\boxtimes F)L(\xi,\widetilde{I})\eta = L(\xi,\widetilde{I})F\eta, \ (F\boxtimes \operatorname{id}_i)R(\xi,\widetilde{I})\eta = R(\xi,\widetilde{I})F\eta.
        \]
        \item (State-field correspondence) For any $\hilb_i\in \operatorname{Obj}(\operatorname{Rep}(\A)), \widetilde{I}\in \widetilde{\mathcal{I}} $ and $\xi\in\hilb_i(I)$,
        we have:
        \[
        L(\xi,\widetilde{I})\Omega = \xi = R(\xi,\widetilde{I})\Omega,
        \]
        by the natural identifications $\hilb_i = \hilb_i\boxtimes\hilb_0 = \hilb_0\boxtimes \hilb_i.$
        \item (Density of fusion product) For any $\hilb_i, \hilb_j \in  \operatorname{Obj}(\operatorname{Rep}(\A)), \widetilde{I}\in \widetilde{\mathcal{I}} ,$
        the set $L(\hilb_i,\widetilde{I})\hilb_j$ spans a dense subspace of $\hilb_i\boxtimes\hilb_j$ and 
        $R(\hilb_i,\widetilde{I})\hilb_j$ spans a dense subspace of  $\hilb_j\boxtimes \hilb_i$.
        \item (Locality) For any $\hilb_k \in  \operatorname{Obj}(\operatorname{Rep}(\A)) $ , disjoint  $\widetilde{I},\widetilde{J}\in \widetilde{\mathcal{I}}$ such that 
        $\widetilde{I}$ is anticlockwise to $\widetilde{J}$ and any $\xi\in\hilb_i(I), \eta\in\hilb_j(J)$, the following diagram commutes adjointly:
        \[
        \begin{tikzcd}
            \hilb_k \arrow{r}{R(\eta,\widetilde{J})} \arrow[swap]{d}{L(\xi,\widetilde{I})}    &\hilb_k\boxtimes \hilb_j \arrow{d}{L(\xi,\widetilde{I})} \\
            \hilb_i\boxtimes\hilb_k     \arrow{r}{R(\eta,\widetilde{J})}      & \hilb_i\boxtimes\hilb_k\boxtimes\hilb_j
        \end{tikzcd}
        \]
        \item (Unitary Braiding) For any $\hilb_i, \hilb_j \in \operatorname{Obj}(\operatorname{Rep}(\A)) ,$ there is a unitary linear map $\mathbb{B}_{i,j}: \hilb_i\boxtimes\hilb_j \to \hilb_j\boxtimes\hilb_i$, such that:
        \[
        \mathbb{B}_{i,j}L(\xi,\widetilde{I})\eta = R(\xi,\widetilde{I})\eta, 
        \]
        whenever $\widetilde{I}\in\widetilde{\mathcal{I}}, \xi\in\hilb_i(I), \eta\in \hilb_j. $

    \end{enumerate}

\end{definition}

If the M\"{o}bius covariant net $\A$ is conformal covariant, then $\mathfrak{E}$ is \textbf{conformal covariant}, which means that  for each $\hilb_i \in \operatorname{Obj}(\operatorname{Rep}(\A)),$
$\widetilde{I} \in \widetilde{\mathcal{I}}$, $\xi\in \hilb_i(I)$ and $g\in \mathcal{G}_c$, 
there is  a  vector $g\xi g^{-1} \in \hilb_i(gI)$ such that, when acting on any $\A$-module, 
\[
L(g\xi g^{-1} , g\widetilde{I})= g L(\xi, \widetilde{I}) g^{-1}, \ R(g\xi g^{-1} , g\widetilde{I})= g R(\xi, \widetilde{I}) g^{-1}.
\]

When $g\in \widetilde{PSU}(1,1)$, we have 
\[
L(g\xi  , g\widetilde{I})= g L(\xi, \widetilde{I}) g^{-1}, \ R(g\xi  , g\widetilde{I})= g R(\xi, \widetilde{I}) g^{-1}.
\]
A categorical extension $\mathfrak{E}$ is called  \textbf{M\"{o}bius covariant} if the above relations are satisfied. 

A M\"{o}bius covariant representation $\hilb_j$ of $\A$ is called \textbf{dualized} if there exists a \textbf{dual object} $\hilb_{\overline{j}} \in \operatorname{Obj}(\operatorname{Rep}(\A))$ and \textbf{evaluations}
$\operatorname{ev}_{j,\overline{j}} \in \operatorname{Hom}_\A(\hilb_j\boxtimes\hilb_{\overline{j}}, \hilb_0)$ and 
$\operatorname{ev}_{\overline{j},j} \in \operatorname{Hom}_\A(\hilb_{\overline{j}}\boxtimes\hilb_j, \hilb_0)$ such that the following conjugate equations hold:
\begin{align*}
    (\operatorname{ev}_{j,\overline{j}}\otimes \operatorname{id}_j)(\operatorname{id}_j \otimes \operatorname{coev}_{\overline{j}, j})  & = \operatorname{id}_j = (\operatorname{id}_j\otimes\operatorname{ev}_{\overline{j},j})(\operatorname{coev}_{j,\overline{j}}\otimes \operatorname{id}_j)  \\
      (\operatorname{ev}_{\overline{j},j}\otimes \operatorname{id}_{\overline{j}})(\operatorname{id}_{\overline{j}} \otimes \operatorname{coev}_{j,\overline{j}})  & = \operatorname{id}_{\overline{j}} = (\operatorname{id}_{\overline{j}}\otimes\operatorname{ev}_{j,\overline{j}})(\operatorname{coev}_{\overline{j},j}\otimes \operatorname{id}_{\overline{j}})
\end{align*}
are satisfied, where
$\operatorname{coev}_{j,\overline{j}} = \operatorname{ev}_{j,\overline{j}}^*$ and 
$\operatorname{coev}_{\overline{j},j} = \operatorname{ev}_{\overline{j},j}^*$. Note that 
if $\hilb_j$ is dualized,  so is $\hilb_{\overline{j}}$. 
There exist positive numbers 
$d_j= d_{\overline{j}}$, called the \textbf{quantum dimensions} of $\hilb_j $ and $\hilb_{\overline{j}}$,
satisfying 
$\operatorname{ev}_{j,\overline{j}}\operatorname{coev}_{j,\overline{j}} = \operatorname{ev}_{\overline{j},j}\operatorname{coev}_{\overline{j},j} = d_j \operatorname{id}_0 = d_{\overline{j}}\operatorname{id}_0 $.

Denote by $\operatorname{Rep}^{\operatorname{d}}(\A)$
the rigid braided $C^*$-tensor category of 
dualized (M\"{o}bius covariant) representations of $\A$. 
See \cite[Section 3]{gui2021bisognano} for more details. 
Moreover, 
$\mathfrak{E}^{\operatorname{d}}  = \left(\A, \operatorname{Rep}^{\operatorname{d}}(\A), \boxtimes , \hilb  \right)$ is a M\"{o}bius covariant categorical extension \cite[Theorem 3.8]{gui2021bisognano}. 

If $\hilb_i, \hilb_j$ are dualized $\A$-modules and 
$F\in \operatorname{Hom}_\A(\hilb_i, \hilb_j)$, 
then there exists a unique $F^{\vee} \in \operatorname{Hom}_\A(\hilb_{\overline{j}}, \hilb_{\overline{i}})$, called the transpose of $F$, satisfying 
\[
\operatorname{ev}_{i,\overline{i}} (\operatorname{id}_i \otimes F^{\vee}) = \operatorname{ev}_{j,\overline{j}}(F \otimes \operatorname{id}_{\overline{j}}).  
\]
Note that  $F^{\vee\vee} = F$. 
Define $\overline{F} = (F^\vee)^* = (F^*)^\vee \in
\operatorname{Hom}_\A(\hilb_{\overline{i}}, \hilb_{\overline{j}})$, called the \textbf{conjugate} 
of $F$. If $G \in \operatorname{Hom}_\A(\hilb_j, \hilb_k)$,
then 
$\overline{GF} = \overline{G} \cdot \overline{F}$.
$F $ is a projection (resp. unitary, an isometry, a
partial isometry) if and only if $F$ is so.

\section{Discrete $C^*$-Frobenius algebras and infinite-index extensions}
In this section, we generalize the notion of  $C^*$-Frobenius algebra in $\operatorname{Rep}(\A)$ to the case where the underlying $\A$-module is a countably infinite direct sum of simple dualized $\A$-modules. As a main result, we show that 
every discrete $C^*$-Frobenius algebra gives rise to, in general, distinct M\"{o}bius covariant extensions of $\A$. 

Throughout this paper, $\Lambda$ is a monoid with at most countably many elements. Denote by $0$ the unit element of $\Lambda$.
For example, $(\Z,+)$ and $(\mathbb{Q},+)$.

\begin{definition}\label{definitionCFro}    
    A \textbf{ discrete $C^\ast$-Frobenius algebra} in Rep$(\A)$ is a tuple 
    \[
     \left(\hilb \cong \bigoplus_{i\in\Lambda}\hilb_i, \{\mu_i^{\mathfrak{L}}\}_{i\in\Lambda}, \{\mu_j^{\mathfrak{R}}\}_{j\in\Lambda }, \{T_k\}_{k\in\Lambda}\right),
    \]
    where each $ \hilb_i$ is a simple object in $\operatorname{Rep}^{\operatorname{d}}(\A), \mu_i^{\mathfrak{L}}\in\operatorname{Hom}_{\A}(\hilb_i\boxtimes \hilb, \hilb), \mu_j^{\mathfrak{R}} \in \operatorname{Hom}_{\A}(\hilb \boxtimes\hilb_j, \hilb), T_k\in\operatorname{Hom}_{\A}(\hilb, \hilb_k)$, for $i,j,k \in\Lambda$ ($\Lambda$ is called to be the \textbf{labeling set} of the discrete $C^*$-Frobenius algebra), such that the following conditions are satisfied: 
   \begin{enumerate}
       \item (\textbf{Unit}):  for any $i,j \in \Lambda$, $T_i T_j^* = \delta_{i,j}\operatorname{id}_i$, $T_i^*T_i$ is a projection and $\sum_{i\in\Lambda} T_i^*T_i = 1$ with respect to the strong operator topology. In addition, the following identities are satisfied:
       \begin{align*}
           \mu^{\mathfrak{L}}_i(\operatorname{id}_i \boxtimes T_0^* )&= T_i^* \\ 
           \mu_j^{\mathfrak{R}}(T_0^* \boxtimes \operatorname{id}_j) &= T^*_j.
       \end{align*}
       \item (\textbf{Associativity and Frobenius condition}): the following diagrams commute adjointly for any $i,j\in\Lambda$: 
       \[\begin{tikzcd}
\hilb_i\boxtimes \hilb \boxtimes \hilb_j \arrow{r}{\mu_i^{\mathfrak{L}}\boxtimes\operatorname{id}_j} \arrow[swap]{d}{\operatorname{id}_i\boxtimes\mu_j^{\mathfrak{R}}} & \hilb\boxtimes\hilb_j \arrow{d}{\mu_j^{\mathfrak{R}}} \\%
\hilb_i\boxtimes \hilb \arrow{r}{\mu_i^{\mathfrak{L}}}& \hilb.\\
\end{tikzcd}
\]
     \end{enumerate}
\end{definition}

The $\A$-module structure on $\hilb \cong \bigoplus_{i\in\Lambda}\hilb_i$ is given by:
$\pi_{\hilb, I}(a) = \sum_{i\in\Lambda} T_j^* \pi_{j,I}(x)T_j $, for any $I \in \mathcal{I}$ and any $a \in \A(I)$.

\begin{remark}
    The ordinary $C^*$-Frobenius algebra \cite[Definition 2.1]{gui2025connes} is a discrete $C^*$-Frobenius algebra whose labeling set $F$ is a finite set. Let 
    $\left(\hilb_a \cong \bigoplus_{i\in F}\hilb_i, \mu , \iota \right)$ be a $C^*$-Frobenius algebra in $\operatorname{Rep}(\A)$, where $F$ is a finite set. Choose
    $\{T_i\}_{i\in F}$ such that $\iota = T_0^*$, for any $i\in F$,  $T_iT_j^* = \delta_{i,j}\operatorname{id}_i$, $T^*_iT_i$ is a projection and $\sum_{i\in F}T^*_iT_i = 1$.
    
    For $i\in F$, define: 
    \begin{align*}
        \mu_i^\mathfrak{L} &= \mu \circ (T^*_i \boxtimes \operatorname{id}_a)\in \operatorname{Hom}_\A(\hilb_i \boxtimes \hilb_a, \hilb_a), \\
        \mu_i^\mathfrak{R} &= \mu \circ(\operatorname{id}_a \boxtimes T_i^*) \in \operatorname{Hom}_\A(\hilb_a \boxtimes \hilb_i, \hilb_a). 
    \end{align*}
    We will show that  $\left(\hilb_a \cong \bigoplus_{i\in F}\hilb_i, \{\mu^\mathfrak{L}_i\}_{i\in F}, \{\mu^\mathfrak{R}_i\}_{i\in F} , \{T_i\}_{i\in F}   \right)$ is a discrete $C^*$-Frobenius algebra. 
    For any $i\in F$, by the unit condition of the ordinary $C^*$-Frobenius algebra, we have the unit condition: 
\begin{align*}
    \mu_i^\mathfrak{L}(\operatorname{id}_i \boxtimes T_0^*) &= \mu\circ (\operatorname{id}_a \boxtimes T_0^*)(T^*_i \boxtimes \operatorname{id}_0) = T_i^* \\
   \mu_i^\mathfrak{R}(T_0^* \boxtimes \operatorname{id}_i) &= \mu \circ (T_0^*\boxtimes \operatorname{id}_a)(\operatorname{id}_0\boxtimes T_i^*) = T_i^*. 
\end{align*}
The associativity and Frobenius condition follows from the adjoint commutativity of the following diagram:
\[
\begin{tikzcd}
    \hilb_i\boxtimes \hilb_a \boxtimes \hilb_j \  \arrow{r}{T_i^*\boxtimes\operatorname{id}\boxtimes\operatorname{id}} \arrow[swap]{d}{\operatorname{id}\boxtimes \operatorname{id} \boxtimes T_j^*} &\hilb_a\boxtimes \hilb_a\boxtimes \hilb_j \arrow{r}{\mu\boxtimes\operatorname{id}} \arrow{d}{\operatorname{id} \boxtimes \operatorname{id} \boxtimes T_j^*}  & \hilb_a\boxtimes \hilb_j \arrow{d}{\operatorname{id}\boxtimes T_j^*}\\
    \hilb_i \boxtimes \hilb_a \boxtimes \hilb_a  \arrow{r}{T_i^*\boxtimes\operatorname{id} \boxtimes \operatorname{id}} \arrow{d}{\operatorname{id}\boxtimes\mu}    &\hilb_a \boxtimes \hilb_a \boxtimes \hilb_a \arrow{r}{\mu\boxtimes\operatorname{id}} \arrow{d}{\operatorname{id}\boxtimes\mu}     & \hilb_a \boxtimes \hilb_a \arrow{d}{\mu} \\
    \hilb_i \boxtimes \hilb_a  \arrow{r}{T_i^* \boxtimes \operatorname{id}}                   &\hilb_a \boxtimes \hilb_a  \arrow{r}{\mu}       & \hilb_a,
\end{tikzcd}
\]
where subscripts on identity morphisms are suppressed for simplicity.. The adjoint commutativity of the lower-right diagram is nothing but the 
associativity and the Frobenius condition of the ordinary $C^*$-Frobenius algebra. It is easy to see the remaining diagrams commute adjointly.

\end{remark}

\begin{lemma}\label{mu_0}
    For any $i\in \Lambda$, we have:
    \begin{align*}
        \mu_0^\mathfrak{L}(\operatorname{id}_0\boxtimes T_i^*)  = T_i^*  =  \mu_0^\mathfrak{R}(T_i^*\boxtimes \operatorname{id}_0). 
    \end{align*}
\end{lemma}
\begin{proof}
By associativity, we have the following commuting diagram: 
 \[\begin{tikzcd}
\hilb_0\boxtimes \hilb \boxtimes \hilb_i\arrow{r}{\mu_0^{\mathfrak{L}}\boxtimes\operatorname{id}_i} \arrow[swap]{d}{\operatorname{id}_0\boxtimes\mu_i^{\mathfrak{R}}} & \hilb\boxtimes\hilb_i \arrow{d}{\mu_i^{\mathfrak{R}}} \\%
\hilb_0\boxtimes \hilb \arrow{r}{\mu_0^{\mathfrak{L}}}& \hilb ,
\end{tikzcd}
\]
which implies 
\[
\mu_i^{\mathfrak{R}}\left( \mu_0^{\mathfrak{L}}(\operatorname{id}_0\boxtimes T_0^*)\boxtimes\operatorname{id}_i\right)  = 
\mu_0^{\mathfrak{L}}\left( \operatorname{id_0}\boxtimes \mu_i^{\mathfrak{R}}(T_0^*\boxtimes\operatorname{id}_i) \right).
\]
Then by the unit condition, we have $\mu_0^{\mathfrak{L}}(\operatorname{id}_0 \boxtimes T_i^*) = T_i^*$. The other equality,  $T_i^*  =  \mu_0^\mathfrak{R}(T_i^*\boxtimes \operatorname{id}_0)$ is proved similarly. 
\end{proof}

Now define two families of spaces of bounded operators on $\hilb$: for any $\widetilde{I} \in \widetilde{\mathcal{I}}$,
\begin{align}{\label{defining}}
    S^{\mathfrak{L}}(\widetilde{I}) &= \operatorname{span}_\C\left\{  \mu_i^{\mathfrak{L}}L(\xi_i,\widetilde{I})\restriction_{\hilb}:\ i\in\Lambda,\  \xi_i\in\hilb_i(I) \right\}, \\
    S^{\mathfrak{R}}(\widetilde{I}) &=  \operatorname{span}_\C\left\{  \mu_i^{\mathfrak{R}}R(\xi_i,\widetilde{I})\restriction_{\hilb}:\ i\in\Lambda,\  \xi_i\in\hilb_i(I) \right\}. 
\end{align}
Denote by $\B^{\mathfrak{L}}(\widetilde{I})$ the von Neumann algebra generated by $S^{\mathfrak{L}}(\widetilde{I})$ and denote by $\B^{\mathfrak{R}}(\widetilde{I})$ the von Neumann algebra generated by
$S^{\mathfrak{R}}(\widetilde{I}).$

\begin{proposition}\label{inclusion}
    $\pi_{\hilb,I}\left(\A(I)\right)\subset \B^{\mathfrak{L}}(\widetilde{I}) \cap \B^{\mathfrak{R}}(\widetilde{I})$, for $\widetilde{I}\in\widetilde{\mathcal{I}}$.
\end{proposition}
\begin{proof}
Fix $\widetilde{I} \in \widetilde{\mathcal{I}}$.  We only show  $\pi_{\hilb,I}\left(\A(I)\right)\subset \B^{\mathfrak{L}}(\widetilde{I})$. $\pi_{\hilb,I}\left(\A(I)\right) \subset \B^{\mathfrak{R}}(\widetilde{I})$ follows similarly. By the unit condition and Lemma \ref{mu_0},  $\pi_{\hilb,I}\left(\A(I)\right)\subset S^{\mathfrak{L}}(\widetilde{I})$. Indeed, for any $x\in\A(I)$ and any $\eta\in\bigoplus^{\operatorname{alg}}_{i\in\Lambda}\hilb_i(I)$, we have:
    \begin{align*}
        \mu_0^{\mathfrak{L}}L(x\Omega,\widetilde{I})\eta &= \sum_{j\in\operatorname{supp}\eta} \mu_0^{\mathfrak{L}} L(x\Omega,\widetilde{I})T_j^*T_j\eta \\
        &= \sum_{j\in\operatorname{supp}\eta}\mu_0^{\mathfrak{L}}(\operatorname{id}_0\boxtimes T_j^*)L(x\Omega,\widetilde{I})T_j\eta \\
        &= \sum_{j\in\operatorname{supp}\eta} T_j^* \pi_{j,I}(x)T_j\eta\\
        &= \sum_{j\in\operatorname{supp}\eta}T_j^*T_j\pi_{\hilb,I}(x)\eta\\
        & = \pi_{\hilb,I}(x)\eta, 
    \end{align*}
    where $\operatorname{supp}\eta = \left\{i\in\Lambda: T_i \eta \neq0 \right\}$  is a finite subset of $\Lambda$.
\end{proof}

\begin{proposition}\label{cyclicvacuum1}
    Assume $\xi\in\bigoplus^{\operatorname{alg}}_{i\in\Lambda}\hilb_i(I)$, the algebraic direct sum of $\left\{\hilb_i(I)\right\}_{i\in\Lambda}$ , for some $I \in \mathcal{I}$. Then,
    \begin{align*}
        \sum_{i\in\operatorname{supp}\xi}\mu_i^{\mathfrak{L}}L(\xi_i,\widetilde{I}) \cdot T_0^*\Omega = \xi =    \sum_{i\in\operatorname{supp}\xi}\mu_i^{\mathfrak{R}}R(\xi_i,\widetilde{I}) \cdot T_0^*\Omega,
    \end{align*}
    where $\operatorname{supp}\xi = \left\{i\in\Lambda: T_i\xi \neq0 \right\}$. 
\end{proposition}
\begin{proof}
    \begin{align*}
        \sum_{i\in\operatorname{supp}\xi}\mu_i^{\mathfrak{L}}L(\xi_i,\widetilde{I}) \cdot T_0^*\Omega &= \sum_{i\in\operatorname{supp}\xi}\mu_i^{\mathfrak{L}}(\operatorname{id}_i \boxtimes T_0^*)L(\xi_i,\widetilde{I})\Omega \\
        &= \sum_{i\in\operatorname{supp}\xi}T_i^*L(\xi_i,\widetilde{I})\Omega  \\
        & = \sum_{i\in\operatorname{supp}\xi} T^*_i\xi_i\\
        & = \xi. 
    \end{align*}
    The other identity follows similarly. 
\end{proof}

If $S$ is a set of bounded linear operators on a Hilbert space $\mathcal{K}$, its commutant $S'$
is defined to be the set of bounded linear operators on $\mathcal{K}$ which \textbf{commute adjointly} with the operators in $S$. Then $S'$
is a von Neumann algebra. The double commutant $S''$ is called the von Neumann algebra generated by $S$.

\begin{proposition}{\label{commutant}}
If $\widetilde{I}$ is anticlockwise to $\widetilde{J}$ in $\widetilde{\mathcal{I}}$, then
$[\B^{\mathfrak{L}}(\widetilde{I}), \B^{\mathfrak{R}}(\widetilde{J})]= 0.$
\end{proposition}
\begin{proof}
   $S^{\mathfrak{L}}(\widetilde{I})' \supset S^{\mathfrak{R}}(\widetilde{I}')$ follows from the adjoint commutativity of the following diagram.
    \[
    \begin{tikzcd}
        \hilb \arrow{r}{L(\xi,\widetilde{I})}  \arrow[swap]{d}{R(\eta,\widetilde{I}')}      &\hilb_i\boxtimes\hilb \arrow{r}{\mu_i^{\mathfrak{L}}} \arrow{d}{R(\eta,\widetilde{I}')} &\hilb \arrow{d}{R(\eta,\widetilde{I}')}\\
        \hilb\boxtimes\hilb_j \arrow{r}{L(\xi,\widetilde{I})} \arrow[swap]{d}{\mu_j^{\mathfrak{R}}}     &\hilb_i\boxtimes\hilb\boxtimes\hilb_j     \arrow{r}{\mu_i^{\mathfrak{L}}\boxtimes\operatorname{id}_j} \arrow{d}{\operatorname{id}_i\boxtimes\mu_j^{\mathfrak{R}}}             &\hilb\boxtimes\hilb_j\arrow{d}{\mu_j^{\mathfrak{R}}} \\
        \hilb  \arrow{r}{L(\xi,\widetilde{I})}                     &\hilb_i\boxtimes\hilb \arrow{r}{\mu_i^{\mathfrak{L}}}   &\hilb 
    \end{tikzcd}
    \]
    The upper-left diagram commutes adjointly by the locality of the categorical extension. The adjoint commutativity of the upper-right and lower-left diagrams follows from the functoriality of the categorical extension. The lower-right diagram commutes adjointly by the associativity and Frobenius condition of the discrete $C^*$-Frobenius algebra in Definition \ref{definitionCFro}.

\end{proof}

\begin{theorem}\label{relative local extensions}
    Let $(\hilb \cong \bigoplus_{i\in\Lambda}\hilb_i, \{\mu_i^{\mathfrak{L}}\}_{i\in\Lambda}, \{\mu_j^{\mathfrak{R}}\}_{j\in\Lambda }, \{T_k\}_{k\in\Lambda})$ be a discrete  $C^*$-Frobenius algebra in $\operatorname{Rep}(\A)$. The two families of von Neumann algebras $\left\{\B^{\mathfrak{L}}(\widetilde{I})\right\}_{\widetilde{I}\in\widetilde{\mathcal{I}}}, \left\{\B^{\mathfrak{R}} (\widetilde{I})\right\}_{\widetilde{I}\in\widetilde{\mathcal{I}}}$  satisfy the following properties:
    \begin{enumerate}[(a)]
        \item (Extension) For any $\widetilde{I} \in \widetilde{\mathcal{I}}$, we have, 
        $\pi_{\hilb, I}(\A(I)) \subset \B^\mathfrak{L}(\widetilde{I})$ and 
        $\pi_{\hilb, I}(\A(I)) \subset \B^\mathfrak{R}(\widetilde{I})$.
        \item (Cyclicity) For any $\widetilde{I}
        \in\widetilde{\mathcal{I}}$, $T_0^*\Omega$ is cyclic for
        $\B^{\mathfrak{L}}(\widetilde{I})$ and $\B^{\mathfrak{R}}(\widetilde{I})$.
        \item (Relative locality) If $\widetilde{I}$ is anticlockwise to $\widetilde{J}$ in $\widetilde{\mathcal{I}}$, then
$[\B^{\mathfrak{L}}(\widetilde{I}), \B^{\mathfrak{R}}(\widetilde{J})]= 0.$
        \item (M\"{o}bius covariance) They are M\"{o}bius covariant. In other words, for $g\in\widetilde{\operatorname{PSU}}(1,1), \widetilde{I} \in\widetilde{\mathcal{I}}$, 
        \[
        g \B^{\mathfrak{L}}(\widetilde{I})g^{-1} =  \B^{\mathfrak{L}}(g\widetilde{I}), 
        \ g \B^{\mathfrak{R}}(\widetilde{I})g^{-1} =  \B^{\mathfrak{R}}(g\widetilde{I}).
        \]
        \item (Conformal covariance) If $\A$ is a conformal net, they are conformal covariant. For each $g\in\mathcal{G}_c$ and $\widetilde{I} \in\widetilde{\mathcal{I}}$, 
        \[
        g \B^{\mathfrak{L}}(\widetilde{I})g^{-1} =  \B^{\mathfrak{L}}(g\widetilde{I}), 
        \ g \B^{\mathfrak{R}}(\widetilde{I})g^{-1} =  \B^{\mathfrak{R}}(g\widetilde{I}).
        \]
    \end{enumerate}
\end{theorem}
\begin{proof}
Extension follows from Proposition \ref{inclusion}.
Cyclicity follows from Proposition \ref{cyclicvacuum1} and the fact that
$\bigoplus^{\operatorname{alg}}_{i\in\Lambda}\hilb_i(I)$ is a dense subspace of $\hilb$. Relative locality follows from Proposition \ref{commutant}.   
The M\"{o}bius covariance or the conformal covariance follows from the categorical extension and the fact that $\left\{\mu_i^\mathfrak{L}\right\}_{i\in\Lambda}$ and $\left\{\mu_i^\mathfrak{R}\right\}_{i\in\Lambda}$ intertwine the action of $\widetilde{PSU}(1,1)$ or $\mathcal{G}_\A$.
\end{proof}

\section{Commutative discrete $C^*$-Frobenius algebras and infinite-index local extensions}

\begin{definition}
    A \textbf{(local) M\"{o}bius covariant extension} of $\A$  is a family of von Neumann algebras $\left\{\B(\widetilde{I}) \right\}_{\widetilde{I}\in\widetilde{\mathcal{I}}}$ on a fixed  Hilbert space
    $\hilb$ such that $\hilb$ is an $\A$-module and the following conditions are satisfied:
    \begin{enumerate}[(A)]
        \item (Extension) For any $\widetilde{I}\in\widetilde{\mathcal{I}},\ \pi_{\hilb,I}\left(\A(I)\right) \subset \B(\widetilde{I}) .$ 
        \item (Isotony) If $\widetilde{I} \subset \widetilde{J}$, then $\B(\widetilde{I}) \subset \B(\widetilde{J}). $ 
         \item (Locality) If $I \cap J = \emptyset$, then $[\B(\widetilde{I}), \B(\widetilde{J})] = \{0\}$.
        \item (M\"{o}bius covariance) There is a strongly continuous unitary representation $U$  of $\widetilde{PSU}(1,1)$ on $\hilb$ such that 
        \[
        U(g)\B(\widetilde{I})U(g)^*  = \B(g\widetilde{I}),
        \]
        for any $g\in \widetilde{PSU}(1,1), \widetilde{I} \in \widetilde{\mathcal{I}}.$
        \item (Positive energy) The generator of the rotation subgroup is positive.
        \item (Vacuum) There is a $\widetilde{PSU}(1,1)$-invariant vector $\widetilde{\Omega}$, called the vacuum vector, such that 
        $\widetilde{\Omega}$ is cyclic for $\bigvee_{\widetilde{I}\in\widetilde{\mathcal{I}}}\B(\widetilde{I}).$
    \end{enumerate}
\end{definition}

\begin{definition}
    A discrete $C^*$-Frobenius algebra is called \textbf{commutative} if 
    for any $i,j \in\Lambda,$
    \[
  \mu_i^{\mathfrak{L}}(\operatorname{id}_i\boxtimes T_j^*)\mathbb{B}_{j,i} = \mu_i^{\mathfrak{R}}(T_j^*\boxtimes\operatorname{id}_i).
     \]
\end{definition}

\begin{theorem}\label{localextension}
    Let $(\hilb \cong \bigoplus_{i\in\Lambda}\hilb_i, \{\mu_i^{\mathfrak{L}}\}_{i\in\Lambda}, \{\mu_j^{\mathfrak{R}}\}_{j\in\Lambda }, \{T_k\}_{k\in\Lambda})$ be a commutative discrete $C^*$-Frobenius algebra. Then $\B^{\mathfrak{L}}(\widetilde{I}) = \B^{\mathfrak{R}}(\widetilde{I})$ for $\widetilde{I} \in \widetilde{\mathcal{I}}$ and $\{\B^{\mathfrak{L}}(\widetilde{I})\}_{\widetilde{I}\in\widetilde{\mathcal{I}}}$ defines a local 
     M\"{o}bius covariant extension of $\A$. 
\end{theorem}
\begin{proof}
By Theorem \ref{relative local extensions}, it suffices to show for any $\widetilde{I}\in\widetilde{\mathcal{I}
}, \B^{\mathfrak{L}}(\widetilde{I}) = \B^{\mathfrak{R}}(\widetilde{I}).$
    Assume $F$ is a finite subset of $\Lambda$. Take  $\eta_j \in \hilb_j(I)$ for any $j\in F$.  Let $\xi\in\bigoplus^{\operatorname{alg}}_{i\in\Lambda}\hilb_i(I)$ so that $\operatorname{supp}\xi$ is a finite subset of $\Lambda$.  Then we have: 
    \begin{align*}
        \sum_{j\in F} \sum_{k\in\operatorname{supp}\xi} \mu_j^{\mathfrak{R}}R(\eta_j,\widetilde{I})T_k^*T_k\xi &=   \sum_{j\in F} \sum_{k\in\operatorname{supp}\xi}\mu_j^{\mathfrak{R}}(T_k^*\boxtimes \operatorname{id}_j) R(\eta_j,\widetilde{I}) T_k\xi \\
        & =  \sum_{j\in F} \sum_{k\in\operatorname{supp}\xi}  \mu_j^{\mathfrak{L}}(\operatorname{id}_j\boxtimes T_k^*)\mathbb{B}_{k,j}R(\eta_j,\widetilde{I}) T_k\xi \\
        & =   \sum_{j\in F} \sum_{k\in\operatorname{supp}\xi} \mu_j^{\mathfrak{L}}(\operatorname{id}_j\boxtimes T_k^*)L(\eta_j,\widetilde{I}) T_k\xi\\
         &= \sum_{j\in F}\sum_{k\in\operatorname{supp}\xi} \mu_j^{\mathfrak{L}}L(\eta_j,\widetilde{I})T_k^*T_k \xi \\
        &= \sum_{j\in F} \mu_j^{\mathfrak{L}}L(\eta_j,\widetilde{I})\xi.
    \end{align*}
    Therefore,  $S^{\mathfrak{L}}(\widetilde{I}) = S^{\mathfrak{R}}(\widetilde{I})$ and $\B^{\mathfrak{L}}(\widetilde{I}) = \B^{\mathfrak{R}}(\widetilde{I})$. 
\end{proof}

\section{Infinite Simple Current Extensions by $\Z$}
Let $\left\{\hilb_n\right\}_{n\in\Z}$ be a sequence of  simple objects of $\operatorname{Rep}^{\operatorname{d}}(\A)$ such that the dual of $\hilb_1$ is $\hilb_{-1}$ and 
$\hilb_n \boxtimes \hilb_m \cong \hilb_{n+m}, $ unitarily as $\A$-modules for $n,m\in\Z$. Assume $\hilb_{n} $ and $\hilb_{m}$ are inequivalent whenever $n\neq m$. 
In addition, the self-braiding of $\hilb_1$ is assumed to be trivial. In other words, $\mathbb{B}_{1,1}: \hilb_1 \boxtimes \hilb_1\to \hilb_1\boxtimes \hilb_1  $ is the identity operator on $\hilb_1 \boxtimes \hilb_1$ and we require the quantum dimension of $\hilb_1$ to be equal to $1$, in other words, $d_1 =1 $. 
Let $\mathcal{C}$ be the (non-trivial) unitary $C^*$-subtensor category generated by $\hilb_1$ in $\operatorname{Rep}(\A)$. 

\begin{lemma}{\label{associativity and Frobenius}}
There exists a family of unitary intertwiners \(V_{i,j}^{i+j}\in \operatorname{Hom}_{\A}(\hilb_i \boxtimes \hilb_j, \hilb_{i+j})\), \(i,j\in\mathbb Z\), 

    such that for all $n,k \in\Z$, the following identities hold: 
    \begin{align}
    V^{i+k+j}_{i+k,j}(V^{i+k}_{i,k}\boxtimes \operatorname{id}_j) &=   V^{i+k+j}_{i, k+j}(\operatorname{id}_i\boxtimes V^{k+j}_{k,j}), \\
    V^{n+k}_{n,k} \circ \mathbb{B}_{k,n} & = V^{n+k}_{k,n}.
\end{align}
\end{lemma}
\begin{proof}
    The result that we can choose a family of unitary intertwiners satisfying (3) stems from  the unitary associative isomorphisms and the fact that 
    $H^3(\Z, U(1))  = 0$.  For $n,k\in\Z$,  $V^{n+k}_{n,k} \circ \mathbb{B}_{k,n}  = \chi(k,n)V^{n+k}_{k,n}$, since the 
    spaces of intertwiners are at most one-dimensional. By associativity and functoriality, $\chi(n,k)$ is a bicharacter of $\Z$. By the triviality of the self-braiding of $\hilb_1$, we have
    $\chi(1,1) = 1$. Therefore, $\chi(n,k) = 1$ for any $n,k\in\Z.$
\end{proof}
Let $\hilb = \bigoplus_{n\in\Z}\hilb_{n}$. For $i,j\in \Z$, define 
\begin{align*}
    \mu_i^{\mathfrak{L}} &= \sum_{j\in\Z}T_{i+j}^*V_{i,j}^{i+j}(\operatorname{id}_i\boxtimes T_j) ,\\
    \mu_j^{\mathfrak{R}} &= \sum_{m\in\Z} T_{j+m}^*V^{j+m}_{m,j}(T_m\boxtimes \operatorname{id}_j).
\end{align*}
The following result is a special case of \cite[Lemma $2.2$]{masuda2000generalization}. For the reader's convenience, we include a proof.
\begin{proposition}\label{bddoperator}
    For $i,j\in \Z$, $\mu_i^{\mathfrak{L}}$ and $\mu_j^{\mathfrak{R}}$ are bounded operators.
\end{proposition}
\begin{proof}
Choose $\eta_i \in \hilb_{i}$ and $\xi\in\hilb$. 
  \begin{align*}
      \| \sum_{j\in\Z}T_{i+j}^*V_{i,j}^{i+j}(\operatorname{id}_i\boxtimes T_j)(\eta_i\boxtimes\xi)\|^2 &= \sum_{j_1,j_2}\langle T_{i+{j_1}}^*V_{i,j_1}^{i+{j_1}}(\operatorname{id}_i\boxtimes T_{j_1})(\eta_i\boxtimes\xi), 
      T_{i+{j_2}}^*V_{i,j_2}^{i+{j_2}}(\operatorname{id}_i\boxtimes T_{j_2})(\eta_i\boxtimes\xi)  \rangle \\
          &= \sum_{j_1,j_2\in\Z} \langle (\operatorname{id}_1\boxtimes T_{j_2}^*)(V^{i+j_2}_{i,j_2})^*T_{i+j_2}T^*_{i+j_1}V^{i+j_1}_{i,j_1}(\operatorname{id}_i\boxtimes T_{j_1})(\eta_i \boxtimes\xi), \eta_i \boxtimes\xi\rangle\\
          &= \sum_{j\in\Z} \langle (\operatorname{id}_1\boxtimes T_{j}^*)(V^{i+j}_{i,j})^* V^{i+j}_{i,j}(\operatorname{id}_i\boxtimes T_{j})(\eta_i\boxtimes\xi) , \eta_i\boxtimes\xi  \rangle \\
          &= \| \eta_i\boxtimes\xi   \|^2.
  \end{align*}

\end{proof}   

\begin{lemma}{\label{adjointly}}
    For $i, j \in \Z$,\ 
    $\mu_j^{\mathfrak{R}}(\mu_i^{\mathfrak{L}}\boxtimes \operatorname{id}_j) = \mu_i^{\mathfrak{L}}(\operatorname{id}_i\boxtimes\mu_j^{\mathfrak{R}}) $
    and 
    $(\mu_i^{\mathfrak{L}}\boxtimes \operatorname{id}_j)(\operatorname{id}_i\boxtimes (\mu_j^{\mathfrak{R}})^*) = (\mu_j^{\mathfrak{R}})^* \mu_i^{\mathfrak{L}}$.
\end{lemma}
\begin{proof}
    
 \begin{align*}
     \mu_j^{\mathfrak{R}}(\mu_i^{\mathfrak{L}}\boxtimes \operatorname{id}_j) &= [\sum_{m\in\Z}T^*_{j+m}V^{m+j}_{m,j}(T_m\boxtimes \operatorname{id}_j) ]
                                                      [\sum_{l\in\Z}T^*_{i+l}V^{i+l}_{i,l}(\operatorname{id}_i\boxtimes T_l)\boxtimes \operatorname{id}_j ]\\
                                                   & = \sum_{m\in\Z} T^*_{j+m}V^{m+j}_{m,j}(V^m_{i,m-i}(\operatorname{id}_i\boxtimes T_{m-i})\boxtimes\operatorname{id_j})\\
     \mu_i^{\mathfrak{L}}(\operatorname{id}_i\boxtimes\mu_j^{\mathfrak{R}})  &= [\sum_{m\in\Z}T^*_{i+m}V^{i+m}_{i,m}(\operatorname{id}_i \boxtimes T_m)]
                                                      [\operatorname{id}_i \boxtimes \sum_{l\in\Z} T_{l+j}^* V^{l+j}_{l,j}(T_l\boxtimes \operatorname{id}_j)] \\
                                                   &= \sum_{m\in\Z}T^*_{i+m} V^{i+m}_{i,m}(\operatorname{id}_i\boxtimes V^m_{m-j,j})(\operatorname{id}_i \boxtimes (T_{m-j}\boxtimes \operatorname{id}_j))\\
                                                   &= \sum_{m\in\Z} T^*_{i+m} V^{i+m}_{m+i-j, j}(V^{i+m-j}_{i, m-j}\boxtimes \operatorname{id}_j)(\operatorname{id}_i\boxtimes (T_{m-j}\boxtimes \operatorname{id}_j))\\
                                                   &= \sum_{m\in\Z} T^*_{i+m} V^{i+m}_{m+i-j, j}(V^{i+m-j}_{i, m-j}(\operatorname{id}_i\boxtimes T_{m-j})\boxtimes \operatorname{id}_j)\\
                                                   &\overset{m \mapsto m-i+j}{=} \sum_{m\in\Z} T^*_{m+j}V^{m+j}_{m,j}(V_{i,m-i}^m(\operatorname{id}_i\boxtimes T_{m-i})\boxtimes \operatorname{id}_j).
 \end{align*}
The remaining identity follows from a similar direct computation, using the unitarity of the intertwiners.
 \end{proof}

\begin{proposition}\label{localbriading}
    $ \mu_i^{\mathfrak{L}}(\operatorname{id}_i\boxtimes T_j^*)\mathbb{B}_{j,i} = \mu_i^{\mathfrak{R}}(T_j^*\boxtimes\operatorname{id}_i)$, for $i,j\in\Z.$
\end{proposition}
\begin{proof}
    \begin{align*}
        \mu_i^{\mathfrak{L}}(\operatorname{id}_i\boxtimes T_j^*)\mathbb{B}_{j,i} &= \sum_{k\in\Z}T_{i+k}^*V_{i,k}^{i+k}(\operatorname{id}_i\boxtimes T_k) (\operatorname{id}_i\boxtimes T_j^*)\mathbb{B}_{j,i}\\
        &=T_{i+j}^*V^{i+j}_{i,j}\mathbb{B}_{j,i}(T_jT_j^*\boxtimes \operatorname{id}_i)  \\ 
        & = T_{i+j}^*V^{i+j}_{j,i}(T_jT_j^*\boxtimes \operatorname{id}_i) \\
        & = \mu_i^{\mathfrak{R}}(T_j^*\boxtimes\operatorname{id}_i).
    \end{align*}
\end{proof}
 
We have two spaces of bounded operators on $\hilb$:
\begin{align*}
    S^{\mathfrak{L}}(\widetilde{I}) &=\operatorname{span}_\C \left\{  \mu_i^{\mathfrak{L}}L(\xi_i,\widetilde{I})\restriction_{\hilb}: \xi_i\in\hilb_{i}(I), i\in \Z\right\}, \\
    S^{\mathfrak{R}}(\widetilde{I}) &= \operatorname{span}_\C \left\{ \mu_i^{\mathfrak{R}}R(\xi_i,\widetilde{I})\restriction_{\hilb}: \xi_i\in\hilb_{i}(I), i \in\Z\right\}.
\end{align*}
Denote by $\B^{\mathfrak{L}}(\widetilde{I})$ the von Neumann algebra generated by $S^{\mathfrak{L}}(\widetilde{I})$ and denote by $\B^{\mathfrak{R}}(\widetilde{I})$ the von Neumann algebra generated by
$S^{\mathfrak{R}}(\widetilde{I}).$
Summarizing the above results, we obtain the following theorem.
\begin{theorem}\label{simplecurrentextension}
   For any $\widetilde{I} \in \widetilde{\mathcal{I}}$, $\B^{\mathfrak{L}}(\widetilde{I}) = \B^{\mathfrak{R}}(\widetilde{I})$. Denote this common von Neumann algebra by $\B(\widetilde{I})$. 
   As a consequence, 
   the family of von Neumann algebras $\left\{ \B(\widetilde{I})\right\}_{\widetilde{I} \in\widetilde{\mathcal{I}}}$ forms a local M\"{o}bius covariant extension of  $\A$.
   Moreover, for any $\widetilde{I}\in\widetilde{\mathcal{I}}$, 
   the von Neumann algebra $\B(\widetilde{I})$ depends only on $I$ and not on the choice of the arg-function.
   The strongly continuous unitary representation
   of $\widetilde{PSU}(1,1) $ on $\hilb$ factors through $PSU(1,1)$.
\end{theorem}
\begin{proof}
    The unit condition follows from the natural identifications: $\hilb_i = \hilb_i\boxtimes\hilb_0 = \hilb_0\boxtimes \hilb_i, i\in\Z$. The associativity and Frobenius conditions follow from Lemma \ref{adjointly}. 
    Proposition \ref{localbriading} verifies the commutativity. The conclusion now follows from Theorem \ref{localextension}
    together with the fact that for any $i\in\Z$, 
    the strongly continuous unitary representation $U_i$ of $\widetilde{PSU}(1,1)$ on $\hilb_i$ factors through $PSU(1,1)$,
    or equivalently, the $2\pi$ rotation acts trivially
    by the spin-statistics theorem \cite{guido1996conformal}. 
\end{proof}

\section{Discrete Longo-Rehren constructions}
We retain the notation and setting of the previous section. 
Denote by $\mathbb{M}$ the two-dimensional Minkowski spacetime equipped with metric $dt^2 - dx^2$
and the light-ray coordinates 
$\xi_{\pm} = t \pm x$. 
We have the decomposition 
$\mathbb{M} = \mathcal{L}_+ \times \mathcal{L}_-$, 
where 
$\mathcal{L}_{\pm} = \left\{\xi\in\mathbb{M} :\xi_{\pm} = 0 \right\}$
are the two light ray lines. The $2$-torus 
$S^1 \times S^1$ is a conformal completion of 
$\mathbb{M} = \mathcal{L}_+ \times \mathcal{L}_-$. More explicitly, $\mathbb{M}$
is conformally diffeomorphic to a dense open subregion of $S^1 \times S^1$ and the local action
of $PSL(2,\R) \times PSL(2,\R) $ on $\mathbb{M}$ extends to a global conformal action on $S^1 \times S^1$.

A \textbf{(local) M\"{o}bius  covariant  net $\mathcal{F}$ on 
$S^1 \times S^1$ } is a family of von Neumann algebras
$\left\{\mathcal{F}(I\times J) \right\}_{I, J \in\mathcal{I}}$ on a fixed Hilbert space $\hilb_0$, 
satisfying the following properties:
\begin{enumerate}
    \item (Isotony) If $I_1 \times J_1 \subset I_2 \times J_2$,
    then $\mathcal{F}(I_1 \times J_1) \subset \mathcal{F}(I_2 \times J_2)$. 
    \item (Locality) If $I_1 \times J_1$ and $I_2 \times J_2$ are disjoint in $\mathcal{I} \times \mathcal{I}$, then
        $[\mathcal{F}(I_1 \times J_1),\mathcal{F}(I_2 \times J_2)] = \{ 0\}$. 
    \item (M\"{o}bius covariance)
    There exists a unitary representation 
    $U$ of $PSU(1,1) \times PSU(1,1) $ on $\hilb_0$ such that 
    for any $I, J \in \mathcal{I}$, 
    \[
    U(g) \mathcal{F}(I\times J) U(g)^*
 = \mathcal{F}(g(I \times J)), \ g \in PSU(1,1) \times PSU(1,1).    \]
    \item (Positive energy) The one-parameter unitary subgroup of $U$ corresponding to
    the rotations has positive generator.
    \item (Vacuum) There exists a $U$-invariant 
    vector $\Omega$ (called the vacuum vector)
    and $\Omega$ is cyclic for $\bigvee_{I,J\in\mathcal{I}}\mathcal{F}(I\times J)$. 
\end{enumerate}

An \textbf{$\mathcal{F}$-module}, or a representation of a 
M\"{o}bius covariant net   $\mathcal{F}$ on $S^1 \times S^1$, 
consists of a Hilbert space $\hilb$ and a family of $*$-representations:
$\pi_{\hilb,I\times J}: \mathcal{F}(I\times J) \to B(\hilb)$, for any $I\times J \in \mathcal{I \times I}$
such that if 
$I\times J \subset K\times L $ in $\mathcal{I\times I}$,
then  $\pi_{\hilb,K\times L}\restriction_{ \mathcal{F}(I\times J)} =  \pi_{\hilb,I\times J}$. We will denote 
$\pi_{\hilb,I\times J}$  by  $\pi_{I\times J}$ whenever
no ambiguity occurs.  Let $\hilb_i$ and $\hilb_j$ be 
$\mathcal{F}$-modules. An intertwiner between $\hilb_i$ and $\hilb_j$ is a bounded operator $T: \hilb_i \to \hilb_j$ such that 
$T\pi_{\hilb_i, I \times J}(a) = \pi_{\hilb_j, I \times J}(a)T$ for any
$a \in \mathcal{F}(I \times J)$, $I,J \in \mathcal{I}$.   We denote the space of intertwiners between $\hilb_i$ and $\hilb_j$ by
$\operatorname{Hom}_{\mathcal{F}}(\hilb_i, \hilb_j)$.

 Denote by $\mathfrak{r}$ the reflection on $S^1$, sending 
 $z \in S^1 $ to $z^{-1} = \overline{z} \in S^1$. 
$\mathfrak{r}\widetilde{I} = ( \mathfrak{r} I, \arg_{\mathfrak{r}I})$, where $\arg_{\mathfrak{r}I}(z) = -\arg_I(\overline{z})$. 
Denote the upper semicircle by $S^1_+$. Define $\widetilde{S^1_+}$ such that $\arg_{\widetilde{S^1_+}}$ takes values in $(0,\pi)$. Define $\widetilde{S^1_-}$ by $\mathfrak{r}\widetilde{S^1_+}$. 
 Let $\Theta = \mathfrak{J}_{\widetilde{S^1_+}}$ be the PCT operator of $\mathfrak{E}^{\operatorname{d}}$, which is an 
 anti-unitary operator from each dualized $\A$-module $\hilb_i$ to its dual object $\hilb_{\overline{i}}$. 
 Denote by $\Theta_i$ the action of $\Theta$ on $\hilb_i \in \operatorname{Obj}(\operatorname{Rep}(\A))$. 
 See \cite[Section 6]{gui2021bisognano} for more details.
 We will frequently use the following result from \cite[Theorem 6.4]{gui2021bisognano}:
 \begin{proposition}[PCT theorem]
     For any $\widetilde{I}\in\widetilde{\mathcal{I}}, $ the following identities hold when acting on any object of $\operatorname{Rep}^{\operatorname{d}}(\A):$
     \[
     \Theta \cdot g \cdot \Theta  = \mathfrak{r}\cdot g \cdot \mathfrak{r}, \ \forall g\in \widetilde{PSU}(1,1). 
     \]
     In addition, for any $\hilb_i \in \operatorname{Obj}\left(\operatorname{Rep}^{\operatorname{d}}(\A)\right)$ and $\xi \in \hilb_i(I), $
     we have: 
     \begin{align*}
         \Theta_i \cdot \hilb_i(I) &=  \hilb_{\overline{i}}(\mathfrak{r}I), \\
         \Theta \cdot L(\xi, \widetilde{I}) \cdot \Theta &=  R(\Theta_i\xi, \mathfrak{r}\widetilde{I}). 
     \end{align*}
 \end{proposition}

By \cite[Theorem 6.4]{gui2021bisognano}, we have:
\begin{proposition}
    Let $\hilb_i , \hilb_j\in \operatorname{Obj}(\operatorname{Rep}^{\operatorname{d}}(\A))$. For any $\widetilde{I}\in\widetilde{\mathcal{I}}, \ \xi\in\hilb_i(I),\  \eta\in\hilb_j, $ we have
    \begin{align*}
            \Theta_{i\boxtimes j}L(\xi, \widetilde{I})\eta &= R(\Theta_i\xi, \mathfrak{r}\widetilde{I})\Theta_j\eta, \\
            \Theta_{j\boxtimes i}R(\xi, \widetilde{I})\eta & = L(\Theta_i\xi, \mathfrak{r}\widetilde{I})\Theta_j\eta.
    \end{align*}
    
\end{proposition}

The unitary braiding is given by:
\[
\mathbb{B}_{\Theta_i\hilb_i, \Theta_j\hilb_j } = \Theta_{i\boxtimes j} \circ \mathbb{B}_{i,j}^{-1} \circ \Theta_{j\boxtimes i}^{-1}\in \operatorname{Hom}(\Theta_i\hilb_i\boxtimes \Theta_j\hilb_j, \Theta_j \hilb_j \boxtimes \Theta_i\hilb_i).
\]

\begin{proposition}\label{braidexchange}
    For $\widetilde{I}\in\widetilde{\mathcal{I}}, \ \xi_i\in\hilb_i(I), \eta_j \in \hilb_j, $ we have: 
    \[
    \mathbb{B}_{\Theta_i\hilb_i, \Theta_j\hilb_j } L(\Theta_i\xi, \mathfrak{r}\widetilde{I}) \Theta_j\eta = 
    R(\Theta_i\xi, \mathfrak{r}\widetilde{I})\Theta_j\eta.
    \]
\end{proposition}
\begin{proof}
    \begin{align*}
         \mathbb{B}_{\Theta_i\hilb_i, \Theta_j\hilb_j } L(\Theta_i\xi, \mathfrak{r}\widetilde{I}) \Theta_j\eta & = \Theta_{i\boxtimes j}\circ \mathbb{B}_{i,j}^{-1}\cdot R(\xi_i,\widetilde{I})\eta_j \\
         &= \Theta_{i\boxtimes j} L(\xi_i,\widetilde{I})\eta_j\\
         &= R(\Theta_i\xi, \mathfrak{r}\widetilde{I})\Theta_j\eta.
    \end{align*}
\end{proof}

Denote the local M\"{o}bius covariant net $\left\{\A(\mathfrak{r}I)\right\}_{I\in\mathcal{I}}$ by $\A^{\mathfrak{r}}$. The M\"{o}bius covariance is implemented by the strongly continuous unitary representation $U^\mathfrak{r}(g) = \mathfrak{r} U(g) \mathfrak{r} = \Theta U(g) \Theta $, for any $g\in{PSU}(1,1)$. 
It follows that $\left\{\A(I) \otimes \A^\mathfrak{r}(J) \right\}_{I,J\in\mathcal{I}}$ forms a local 
M\"{o}bius covariant net on $S^1 \times S^1$. 
For any $n\in\Z$, $\Theta_n\hilb_n$ can be naturally equipped 
with an $\A^\mathfrak{r}$-module structure $\pi^{\mathfrak{r}}_{n, I}: \A^\mathfrak{r}(I) \to B(\Theta_n\hilb_n) $ by:
\[
\pi^{\mathfrak{r}}_{n, I}(x) = \Theta_n \cdot \pi_{n,\mathfrak{r}I}(x) \cdot \Theta^{-1}_n, 
\]
for $I \in \mathcal{I}, \ x\in \A(\mathfrak{r}I) $. Moreover, 
$\hilb_n \otimes \Theta_n \hilb_n$ together with 
$\pi_{n, I}\otimes \pi_{n, J}^\mathfrak{r}$ for any $I, J \in \mathcal{I}$, 
form an $\A \otimes \A^\mathfrak{r}$-module.

Define $\hilb_{\bullet} = \bigoplus_{n\in\Z} \hilb_n\otimes \Theta_n\cdot\hilb_n$.
$\hilb_\bullet$ is naturally equipped with an $\A \otimes \A^\mathfrak{r}$-module structure,
implemented by: 
\[
        \pi_{\hilb_\bullet, I\times J}(x\otimes y ) =  \sum_{j\in\Z}\widetilde{T_{j}}^* (\pi_{j,I}(x) \otimes \pi_{j,J}^{\mathfrak{r}}(y))\widetilde{T_{j}},
\]
for any $I , J \in \mathcal{I}$ and $x\in\A(I)$, $y \in\A^\mathfrak{r}(J)$. 
Choose $\widetilde{T_k}\in \operatorname{Hom}_{\A\otimes\A^\mathfrak{r}}(\hilb_{\bullet}, \hilb_k\otimes\Theta_k\hilb_k)$, $k\in\Z$, such that for any $i,j\in\Z$, 
$\widetilde{T_i}\widetilde{T_j}^* = \delta_{i,j}(\operatorname{id}_i\otimes \operatorname{id}_{\overline{i}})$, $\widetilde{T_i}^*\widetilde{T_i}$ is a projection and $\sum_{i\in\Z}\widetilde{T_i}^*\widetilde{T_i} = \operatorname{id}_{\hilb_\bullet}$ with respect to the strong operator topology.

\begin{lemma}
    For any $i\in\Z,$ $\mu_i^{\mathfrak{L}}$ and $\mu_i^{\mathfrak{R}}$, defined below,  are bounded operators on $\hilb_\bullet$.
\end{lemma}
\begin{proof}
    This follows from the same direct computation as in  Proposition \ref{bddoperator}.  
\end{proof}

In  Lemma \ref{associativity and Frobenius},
we have fixed unitary intertwiners $V_{i,j}^{i+j}\in \operatorname{Hom}_{\A}(\hilb_i \boxtimes \hilb_j, \hilb_{i+j})$  for all $i,j\in\Z$. 
For any $i, j\in\Z$, define 
\begin{align*}
  \widetilde{V^{i+j}_{i,j}} &= V^{i+j}_{i,j}\otimes \overline{V^{i+j}_{i,j}}\\
  &= V^{i+j}_{i,j} \otimes \Theta_{i+j} \cdot V^{i+j}_{i,j} \cdot \Theta^{-1}_{i\boxtimes j},
\end{align*}
 which is a unitary intertwiner in 
$\operatorname{Hom}_{\A\otimes\A^\mathfrak{r}}\left((\hilb_i\otimes\Theta_i\hilb_i)\boxtimes (\hilb_j\otimes\Theta_j\hilb_j), \hilb_{i+j}\otimes\Theta_{i+j}\hilb_{i+j}\right)$.

For $\widetilde{I}, \widetilde{J} \in \widetilde{\mathcal{I}}$, two spaces of bounded operators on $\hilb_\bullet$ are defined as:
\begin{align*}
    S^{\mathfrak{L}}(\widetilde{I}\times\widetilde{J}) = \operatorname{span}_\C \bigg\{ \mu_i^{\mathfrak{L}} &\sum_{i_k=0}^{N(i)} \left(L(\xi_i^{(i_k)},\widetilde{I})\otimes \Theta L(\eta_i^{(i_k)},\mathfrak{r}\widetilde{J})\Theta^{-1}\right)\restriction_{\hilb_{\bullet}}: \\
    \  &\xi_i^{(i_k)}\in\hilb_i(I), \eta_i^{(i_k)}\in\hilb_i(\mathfrak{r}J), \ i\in\Z, 
     \ N(i)\in\Z_{\geq 0},0\leq i_k\leq N(i) \bigg\}, \\
  S^{\mathfrak{R}}(\widetilde{I}\times\widetilde{J}) = \operatorname{span}_\C \bigg\{ \mu_i^{\mathfrak{R}} &\sum_{i_k=0}^{N(i)} \left(R(\xi_i^{(i_k)},\widetilde{I})\otimes \Theta R(\eta_i^{(i_k)},\mathfrak{r}\widetilde{J})\Theta^{-1}\right)\restriction_{\hilb_{\bullet}}: \\
    \  &\xi_i^{(i_k)}\in\hilb_i(I), \eta_i^{(i_k)}\in\hilb_i(\mathfrak{r}J), \ i\in\Z, 
     \ N(i)\in\Z_{\geq 0},0\leq i_k\leq N(i) \bigg\},
\end{align*}
where 
\begin{align*}
    \mu_i^{\mathfrak{L}} &= \sum_{j\in\Z}\widetilde{T_{i+j}}^*\widetilde{V^{i+j}_{i,j}}(\operatorname{id}_{\hilb_i\otimes \Theta_i\hilb_i}\boxtimes \widetilde{T_j}) \\
    \mu_i^{\mathfrak{R}} &= \sum_{j\in\Z}\widetilde{T_{i+j}}^*\widetilde{V^{i+j}_{j,i}}(\widetilde{T_j}\boxtimes\operatorname{id}_{\hilb_i\otimes \Theta_i\hilb_i})
\end{align*}
and  $i \in\Z.$
Denote by $\B^{\mathfrak{L}}(\widetilde{I}\times \widetilde{J})$ the von Neumann algebra generated by $S^{\mathfrak{L}}(\widetilde{I}\times\widetilde{J})$ and denote by $\B^{\mathfrak{R}}(\widetilde{I}\times\widetilde{J})$ the von Neumann algebra generated by
$S^{\mathfrak{R}}(\widetilde{I}\times \widetilde{J}).$

\begin{proposition}\label{extension}
For any $\widetilde{I}, \widetilde{J} \in\widetilde{\mathcal{I}}$, we have 
    $ \pi_{\hilb_\bullet, I \times J}\left(\A(I)\otimes\A^\mathfrak{r}(J)\right) \subset  \B^{\mathfrak{L}}(\widetilde{I}\times\widetilde{J}) \cap \B^{\mathfrak{R}}(\widetilde{I}\times\widetilde{J})  $. 
\end{proposition}
\begin{proof}
Fix a vector $\xi\in\hilb_\bullet$ such that $\operatorname{supp}\xi = \left\{ i\in\Z: \widetilde{T_{i}}\xi \neq 0 \right\}$ is a finite subset of $\Z$ and 
$\xi =\sum_{j\in \operatorname{supp}\xi} \sum_{k =1}^{\xi(j)}\eta_j^{(j_k)} \otimes \zeta_j^{(j_k)}$,
where for any $j\in\operatorname{supp}\xi$,
$\xi(j)\in\Z_{>0}$ and for any $1 \leq k \leq \xi(j)$,
$\eta_j^{(j_k)}, \zeta_j^{(j_k)} \in \hilb_j$.
For $\widetilde{I}, \widetilde{J} \in\widetilde{\mathcal{I}}$, $x\in \A(I)$ and $y \in \A^\mathfrak{r}(J)$, we have: 
    \begin{align*}
        \mu_0^{\mathfrak{L}}(L(x\Omega, \widetilde{I})\otimes \Theta L(y\Omega, \mathfrak{r}\widetilde{J})\Theta^{-1})\xi&=  \sum_{j\in\operatorname{supp}\xi}\sum_{k =1}^{\xi(j)}\widetilde{T_{j}}^*\widetilde{V^{j}_{0,j}}[(L(x\Omega, \widetilde{I})\eta_j^{(j_k)})\otimes (\Theta L(y\Omega, \mathfrak{r}\widetilde{J})\Theta^{-1}\Theta_j\zeta_j^{(j_k)})] \\
        & = \sum_{j\in\operatorname{supp}\xi}\sum_{k =1}^{\xi(j)}\widetilde{T_{j}}^*(V^j_{0,j}L(x\Omega, \widetilde{I})\eta_j^{(j_k)})\otimes(\overline{V^j_{0,j}}\Theta L(y\Omega, \mathfrak{r}\widetilde{J})\zeta_j^{(j_k)})\\
        & = \sum_{j\in\operatorname{supp}\xi}\sum_{k =1}^{\xi(j)}\widetilde{T_{j}}^* (V^j_{0,j}L(x\Omega, \widetilde{I})\eta_j^{(j_k)})\otimes (\Theta_jV^j_{0,j}(L(y\Omega, \mathfrak{r}\widetilde{J})\zeta_j^{(j_k)})) \\
        & = \sum_{j\in\operatorname{supp}\xi}\sum_{k =1}^{\xi(j)}\widetilde{T_{j}}^*  (\pi_{j,I}(x)\eta_j^{(j_k)} \otimes \Theta_j \cdot \pi_{j,\mathfrak{r}J}(y) \cdot \Theta_j^{-1} \cdot \Theta_j \zeta_j^{(j_k)}) \\
        & =  \left(\sum_{j\in\operatorname{supp}\xi}\widetilde{T_{j}}^* (\pi_{j,I}(x) \otimes \pi_{j,J}^{\mathfrak{r}}(y))\widetilde{T_{j}}\right)\xi \\
        &= \pi_{\hilb_\bullet, I\times J}(x\otimes y)\xi.
    \end{align*}
    By a  similar argument, we obtain $\pi_{\hilb_\bullet, I \times J}(\A(I)\otimes\A^\mathfrak{r}(J)) \subset  \B^{\mathfrak{R}}(\widetilde{I}\times\widetilde{J}) $ .
\end{proof}

\begin{proposition}\label{cyclicvacuum}
    Let $\zeta \in \hilb_\bullet$ such that 
    $\widetilde{T_{i}}\zeta  = \sum^{\zeta(i)}_{{k}=1}\xi_i^{(i_k)}\otimes \Theta_i\eta_i^{(i_k)}$, where 
    $i\in\Z, \zeta(i)\in \Z_{> 0}$ and 
    $\xi_i^{(i_k)}  \in \hilb_i(I),\  \eta_i^{(i_k)}\in\hilb_i(\mathfrak{r}J) $, for $1 \leq k \leq \zeta(i)$, 
    where $\widetilde{I}, \widetilde{J}\in\widetilde{\mathcal{I}}$.
Moreover, assume that $\operatorname{supp}\zeta $ is a finite subset of $\Z$.
   Then, 
    \begin{align*}
         \zeta &=  \sum_{i\in\operatorname{supp}\zeta}\mu_i^{\mathfrak{L}}\left( \sum^{\zeta(i)}_{{k}=1}L(\xi_i^{(i_k)},\widetilde{I})\otimes \Theta L(\eta_i^{(i_k)}, \mathfrak{r}\widetilde{J})\Theta^{-1}\right)\widetilde{T_0}^* (\Omega \otimes \Omega) \\
         &=  \sum_{i\in\operatorname{supp}\zeta}\mu_i^{\mathfrak{R}}\left(\sum_{k=1}^{\zeta(i)}R(\xi_i^{(i_k)},\widetilde{I})\otimes \Theta R(\eta_i^{(i_k)}, \mathfrak{r}\widetilde{J})\Theta^{-1} \right) \widetilde{T_0}^* (\Omega \otimes \Omega). 
    \end{align*}

\end{proposition}
\begin{proof}

\begin{align*}
    \sum_{i\in\operatorname{supp}\zeta}\mu_i^{\mathfrak{L}}\left(\sum^{\zeta(i)}_{{k}=1}L(\xi_i^{(i_k)},\widetilde{I})\otimes \Theta L(\eta_i^{(i_k)}, \mathfrak{r}\widetilde{J})\Theta^{-1} \right) \widetilde{T_0}^* (\Omega \otimes \Omega) &=  \sum_{i\in\operatorname{supp}\zeta}\mu_i^{\mathfrak{L}}((\operatorname{id}_i\otimes \operatorname{id}_0)\boxtimes \widetilde{T_0}^*\boxtimes (\operatorname{id}_0\otimes\operatorname{id}_{\overline{i}})) \cdot\\
    &(\sum^{\zeta(i)}_{{k}=1} L(\xi_i^{(i_k)},\widetilde{I})\otimes \Theta L(\eta_i^{(i_k)}, \mathfrak{r}\widetilde{J})\Theta^{-1})(\Omega\otimes\Omega)\\
    &= \sum_{i\in\operatorname{supp}\zeta} \widetilde{T_i}^* \widetilde{V^i_{i,0}}(\sum^{\zeta(i)}_{{k}=1}L(\xi_i^{(i_k)},\widetilde{I})\Omega \otimes L(\eta_i^{(i_k)}, \mathfrak{r}\widetilde{J})\Omega)\\
    &= \sum_{i\in\operatorname{supp}\zeta} \widetilde{T_i}^*(\sum^{\zeta(i)}_{{k}=1}V^i_{i,0}L(\xi_i^{(i_k)},\widetilde{I})\Omega\otimes \Theta_i V^i_{i,0}L(\eta_i^{(i_k)},\mathfrak{r}\widetilde{J})\Omega) \\
    &= \sum_{i\in\operatorname{supp}\zeta} \widetilde{T_i}^*(\sum^{\zeta(i)}_{{k}=1} \xi_i^{(i_k)}\otimes \Theta_i\eta_i^{(i_k)})\\
    & = \zeta.
\end{align*}
 The other identity follows similarly. 
\end{proof}

\begin{lemma}\label{assoFroLR1}
    The unitary intertwiners satisfy the following identities:
    \begin{align*}
        \widetilde{V^{i+k+j}_{i+k,j}}(\widetilde{V^{i+k}_{i,k}}\boxtimes \operatorname{id}_{\hilb_j\otimes\Theta_j\hilb_j}) &=   \widetilde{V^{i+k+j}_{i, k+j}}(\operatorname{id}_{\hilb_i\otimes\Theta_i\hilb_i}\boxtimes \widetilde{V^{k+j}_{k,j})} \\
         \widetilde{V^{i+j}_{i,j}} \widetilde{\mathbb{B}_{j,i}} &= \widetilde{V^{i+j}_{j,i}},
    \end{align*}
  where $ \widetilde{\mathbb{B}_{j,i}} = \mathbb{B}_{j,i} \otimes  \mathbb{B}_{\Theta_i\hilb_i, \Theta_j\hilb_j }, \ i, j , k \in\Z.$

\end{lemma}
\begin{proof}
    The result follows from Lemma \ref{associativity and Frobenius} and properties of conjugate operators. 
\end{proof}
Combining  Lemma \ref{assoFroLR1} and Lemma \ref{adjointly}, we get the following lemma. 
\begin{lemma}\label{assoFroLR2}
For $i, j \in \Z$, we have: 
\begin{align*}
    \mu_j^{\mathfrak{R}}(\mu_i^{\mathfrak{L}}\boxtimes \operatorname{id}_{{\hilb_j\otimes \Theta_j\hilb_j}}) &= \mu_i^{\mathfrak{L}}(\operatorname{id}_{{\hilb_i\otimes \Theta_i\hilb_i}}\boxtimes\mu_j^{\mathfrak{R}}) \\
    (\mu_i^{\mathfrak{L}}\boxtimes \operatorname{id}_{\hilb_j\otimes\Theta_j\hilb_j})(\operatorname{id}_{\hilb_i\otimes\Theta_i\hilb_i}\boxtimes (\mu_j^{\mathfrak{R}})^*) &= (\mu_j^{\mathfrak{R}})^* \mu_i^{\mathfrak{L}}.
\end{align*}
    
\end{lemma}

\begin{proposition}\label{locality1}
    Let $\widetilde{I_1}, \widetilde{J_1}, \widetilde{I_2}, \widetilde{J_2}\in \widetilde{\mathcal{I}}$. If $\widetilde{I_1}$ is anticlockwise to  $\widetilde{I_2}$ and $\widetilde{J_1} $ is clockwise to $ \widetilde{J_2}$ , then  
    \[
   [ \B^{\mathfrak{L}}(\widetilde{I_1}\times\widetilde{J_1}), \B^{\mathfrak{R}}(\widetilde{I_2}\times\widetilde{J_2})    ] =\{0\} .
    \]
    
\end{proposition}
\begin{proof}
Following the proof of Proposition \ref{commutant} and  combining the locality of the categorical extension with  Lemma \ref{assoFroLR1} and Lemma \ref{assoFroLR2}, 
we obtain: 
\[
S^{\mathfrak{R}}(\widetilde{I}'\times \mathfrak{r}\widetilde{J}') \subset S^{\mathfrak{L}}(\widetilde{I}\times \mathfrak{r}\widetilde{J})'.
\]
    
\end{proof}

\begin{lemma}\label{commuLR}
  For any $i,j \in \Z$, we have  
  \[
  \mu_i^{\mathfrak{L}}(\operatorname{id}_{\hilb_i\otimes \Theta_i\hilb_i}\boxtimes\widetilde{T_j}^*)\widetilde{\mathbb{B}_{j,i}} = \mu_i^{\mathfrak{R}}(\widetilde{T_j}^*\boxtimes\operatorname{id}_{\hilb_i\otimes \Theta_i\hilb_i}).
  \]
\end{lemma}
\begin{proof}

\begin{align*}
    \mu_i^{\mathfrak{L}}(\operatorname{id}_{\hilb_i\otimes \Theta_i\hilb_i}\boxtimes\widetilde{T_j}^*)\widetilde{\mathbb{B}_{j,i}}  &=\widetilde{T_{i+j}}^*\widetilde{V^{i+j}_{i,j}}(\operatorname{id}_{\hilb_i\otimes \Theta_i\hilb_i}\boxtimes\widetilde{T_j}\widetilde{T_j}^*)\widetilde{\mathbb{B}_{j,i}} \\
    & = \widetilde{T_{i+j}}^*\widetilde{V^{i+j}_{i,j}}\widetilde{\mathbb{B}_{j,i}} (\widetilde{T_j}\widetilde{T_j}^*\boxtimes\operatorname{id}_{\hilb_i\otimes \Theta_i\hilb_i})\\
    &=  \widetilde{T_{i+j}}^*\widetilde{V^{i+j}_{j,i}}(\widetilde{T_j}\widetilde{T_j}^*\boxtimes\operatorname{id}_{\hilb_i\otimes \Theta_i\hilb_i})\\
    & = \mu_i^{\mathfrak{R}}(\widetilde{T_j}^*\boxtimes\operatorname{id}_{\hilb_i\otimes \Theta_i\hilb_i}).
\end{align*}
    
\end{proof}

\begin{lemma}\label{BR LR}
    For $i,j\in\Z$, $\widetilde{I}, \widetilde{J} \in \widetilde{\mathcal{I}} $ and $\xi_i\in\hilb_i(I),\ \eta_i \in \hilb_i(\mathfrak{r}\widetilde{J}), \ \alpha \in \hilb_j, \ \beta \in \Theta_j\hilb_j,$
we have: 
\[
\widetilde{\mathbb{B}_{j,i}} \circ (R(\xi_i,\widetilde{I})\otimes L(\Theta_i\eta_i,\widetilde{J}))(\alpha\otimes\beta)) = 
                                         L(\xi_i,\widetilde{I})\otimes R(\Theta_i\eta_i,\widetilde{J})(\alpha\otimes\beta).
\]
\end{lemma}
\begin{proof}
By functoriality of the categorical extension and Proposition \ref{braidexchange}, we have: 
    \begin{align*}
       \widetilde{\mathbb{B}_{j,i}} \circ (R(\xi_i,\widetilde{I})\otimes L(\Theta_i\eta_i,\widetilde{J}))(\alpha\otimes\beta)) &= (\mathbb{B}_{j,i}\otimes \mathbb{B}_{\Theta_i\hilb_i, \Theta_j\hilb_j})(R(\xi_i,\widetilde{I})\alpha \otimes L(\Theta_i\eta_i,\widetilde{J})\beta)   \\
       & = L(\xi_i,\widetilde{I})\otimes R(\Theta_i\eta_i,\widetilde{J}) (\alpha \otimes \beta).
    \end{align*}
\end{proof}

Combining the last two lemmas, we have the following: 
\begin{proposition}\label{locality2}
    For $\widetilde{I}, \widetilde{J} \in \widetilde{\mathcal{I}}$,   
   we have: 
    \[
    \B^{\mathfrak{L}}(\widetilde{I}\times\widetilde{J}) = \B^{\mathfrak{R}}(\widetilde{I}\times\widetilde{J}).
    \]
    Denote this common von Neumann algebra by $\B(\widetilde{I}\times\widetilde{J}).$
\end{proposition}
\begin{proof}
Fix $\widetilde{I}, \widetilde{J} \in \widetilde{\mathcal{I}} $. 
Take $\zeta \in \bigoplus^{\operatorname{alg}}_{j\in\Z}\hilb_j\otimes\Theta_j\hilb_j$ so that $\operatorname{supp}\zeta$ is  a finite subset of $\Z$. 
Let  $N(i) \in \Z_{\geq 1}$ for any $i\in\Z$ and choose $\xi_i^{(i_k)} \in \hilb_i(I) $, 
$\eta_i^{(i_k)} \in \hilb_i(J)$, where $k= 1, \cdots, N(i) $. 
 Then,  by Lemma \ref{commuLR} and Lemma \ref{BR LR}, for any fixed $i\in\Z$,  we have: 
\begin{align*}
      \mu_i^{\mathfrak{R}}  \sum_{k=1}^{N(i)} R(\xi_i^{(i_k)},\widetilde{I})\otimes L(\Theta\eta_i^{(i_k)},\mathfrak{r}\widetilde{J}) \sum_{j\in\operatorname{supp}\zeta}\widetilde{T_j}^*\widetilde{T_j}\zeta
    &=  \sum_{j\in\operatorname{supp}\zeta}\mu_i^{\mathfrak{R}}  (\widetilde{T_j}^*\boxtimes\operatorname{id}_{\hilb_i\otimes\Theta_i\hilb_i})\sum_{k=1}^{N(i)} R(\xi_i^{(i_k)},\widetilde{I})\otimes L(\Theta\eta_i^{(i_k)},\mathfrak{r}\widetilde{J})\widetilde{T_j } \zeta \\
    & =   \sum_{j\in\operatorname{supp}\zeta} \mu_i^{\mathfrak{L}}(\operatorname{id}_{\hilb_i\otimes \Theta_i\hilb_i}\boxtimes\widetilde{T_j}^*)\widetilde{\mathbb{B}_{j,i}} \sum_{k=1}^{N(i)} R(\xi_i^{(i_k)},\widetilde{I})\otimes L(\Theta\eta_i^{(i_k)},\mathfrak{r}\widetilde{J})\widetilde{T_j } \zeta \\
    &= \sum_{j\in\operatorname{supp}\zeta} \mu_i^{\mathfrak{L}}(\operatorname{id}_{\hilb_i\otimes \Theta_i\hilb_i}\boxtimes\widetilde{T_j}^*)\sum_{k=1}^{N(i)} L(\xi_i^{(i_k)},\widetilde{I})\otimes R(\Theta\eta_i^{(i_k)},\mathfrak{r}\widetilde{J})\widetilde{T_j } \zeta \\
    & = \mu_i^{\mathfrak{L}}  \sum_{k=1}^{N(i)} L(\xi_i^{(i_k)},\widetilde{I})\otimes R(\Theta\eta_i^{(i_k)},\mathfrak{r}\widetilde{J}) \sum_{j\in\operatorname{supp}\zeta}\widetilde{T_j}^*\widetilde{T_j} \zeta \\
    & = \mu_i^{\mathfrak{L}}  \sum_{k=1}^{N(i)} L(\xi_i^{(i_k)},\widetilde{I})\otimes R(\Theta\eta_i^{(i_k)},\mathfrak{r}\widetilde{J}) \zeta,
\end{align*}    
implying $S^\mathfrak{L}(\widetilde{I}\times\widetilde{J}) = S^\mathfrak{R}(\widetilde{I}\times\widetilde{J})$ and 
$\B^{\mathfrak{L}}(\widetilde{I}\times\widetilde{J}) = \B^{\mathfrak{R}}(\widetilde{I}\times\widetilde{J})$. 
\end{proof}

\begin{theorem}\label{discreteLRconstruction}
    The family of  von Neumann algebras $\left\{ \B(\widetilde{I} \times \widetilde{J} )\right\}_{\widetilde{I},\ \widetilde{J} \in \widetilde{\mathcal{I}}}$  satisfies the following properties:  
    \begin{enumerate}
        \item (Extension) For any $\widetilde{I}, \widetilde{J} \in \widetilde{\mathcal{I}}$, 
        we have  $ \pi_{\hilb_\bullet, I \times J}(\A(I)\otimes\A^\mathfrak{r}(J)) \subset  \B(\widetilde{I}\times\widetilde{J})  $. 
        \item (Isotony) If $\widetilde{I_1} \subset \widetilde{I_2}$ and $\widetilde{J_1} \subset \widetilde{J_2}$, then $\B(\widetilde{I_1} \times \widetilde{J_1}) \subset \B(\widetilde{I_2} \times \widetilde{J_2})$.
        \item (M\"{o}bius covariance) There is a strongly continuous unitary representation $U$
of 
        $\widetilde{PSU}(1,1) \times \widetilde{PSU}(1,1) $ on $\hilb_\bullet$ such that for any $(g_1,g_2)\in \widetilde{PSU}(1,1)\times \widetilde{PSU}(1,1)$ and $\widetilde{I}, \widetilde{J}\in\widetilde{\mathcal{I}}$, 
        \[
        U(g_1, g_2) \B(\widetilde{I}\times \widetilde{J} ) U(g_1, g_2) ^* = \B(g_1\widetilde{I}\times g_2\widetilde{J} ).
        \]
        \item (Positive energy) The generator of the one-parameter rotation subgroup is positive.
        \item (Locality)  Let $\widetilde{I_1}, \widetilde{J_1}, \widetilde{I_2}, \widetilde{J_2}\in \widetilde{\mathcal{I}}$. If $\widetilde{I_1}$ is anticlockwise to  $\widetilde{I_2}$ and $\widetilde{J_1} $ is clockwise to $ \widetilde{J_2}$ , then  
    \[
   [ \B(\widetilde{I_1}\times\widetilde{J_1}), \B(\widetilde{I_2}\times\widetilde{J_2})    ] =\{0\}.
    \]
        \item (Vacuum) $\widetilde{T_0}^*(\Omega\otimes\Omega)$ is a $\widetilde{PSU}(1,1) \times \widetilde{PSU}(1,1)$ invariant vector and is cyclic for $\B(\widetilde{I}\times \widetilde{J} )$, for any  $\widetilde{I} , \widetilde{J} \in \widetilde{\mathcal{I}}.$
    \end{enumerate}
    In addition, for any $\widetilde{I}, \widetilde{J} \in\widetilde{\mathcal{I}}$, the von Neumann algebra $\B(\widetilde{I}\times\widetilde{J})$ depends only on $I, J$ and not on the choices of the arg-functions.  The strongly continuous unitary representation $U$ of $\widetilde{PSU}(1,1) \times \widetilde{PSU}(1,1) $ on $\hilb_\bullet$ factors through $PSU(1,1) \times PSU(1,1)$. 
\end{theorem}
\begin{proof}
Extension follows from Proposition \ref{extension}. 
    Positive energy follows from the positive-energy property of each component representation. The cyclicity of $\widetilde{T_0}^*(\Omega\otimes\Omega)$ follows from Proposition \ref{cyclicvacuum}. Locality follows from Proposition \ref{locality1} and Proposition \ref{locality2}.
    The remaining assertions follow by arguments similar to those used in the proof of Theorem \ref{simplecurrentextension}.
\end{proof}

\printbibliography[
heading=bibintoc,
title={References}
] 

\end{document}